\documentclass[11pt]{amsart}
\usepackage{amssymb}
\usepackage[margin=1in]{geometry}
\usepackage[colorlinks]{hyperref}
\newtheorem{theorem}{Theorem}[section]
\newtheorem{lem}[theorem]{Lemma}
\newtheorem{prop}[theorem]{Proposition}
\newtheorem{cor}[theorem]{Corollary}

\theoremstyle{definition}
\newtheorem{definition}[theorem]{Definition}
\newtheorem{example}[theorem]{Example}

\theoremstyle{remark}
\newtheorem{remark}[theorem]{Remark}

\newtheorem{note}[theorem]{Note}
\numberwithin{equation}{section}

\begin{document}

\newcommand{\spacing}[1]{\renewcommand{\baselinestretch}{#1}\large\normalsize}
\spacing{1.14}

\title{comparison geometry for an extension of Ricci tensor}

\author [S. H. Fatemi and S. Azami]{Seyed Hamed  Fatemi and Shahroud Azami}

\address{Department of Mathematics\\ Tarbiat modares university\\Tehran\\ Iran} \email{fatemi.shamed@gmail.com}
\address{Department of Mathematics, Faculty of Sciences\\ Imam Khomeini International University\\ Qazvin\\ Iran} \email{azami@sci.ikiu.ac.ir}


\date{\today}

\begin{abstract}
For a complete Riemannian manifold $M$ with an (1,1)-elliptic Codazzi self-adjoint tensor field $A$ on it, we use the divergence type operator ${L_A}(u): = div(A\nabla u)$ and an extension of the Ricci tensor to extend some major comparison theorems in Riemannian geometry. In fact we extend theorems like mean curvature comparison theorem,  Bishop-Gromov volume comparison theorem, Cheeger-Gromoll splitting theorem and some of their famous topological consequences. Also we get an upper bound for the end of manifolds by restrictions on the extended Ricci tensor. The results can be applicable for some kind of Riemannian hypersurfaces when the ambient manifold is Riemannian or Lorentzian with constant sectional curvature.
\end{abstract}

\maketitle

\section{Introduction}In comparison geometry, one of the most important theorems is the Laplacian comparison theorem for distance function in complete Riemannian manifolds. This theorem states that for a complete Riemnnian manifold $M$ with $Ri{c_M} \ge (n - 1)H$, we have ${\Delta _M}r \le {\Delta _H}r$, where ${\Delta _M}r$ is the Laplacian of the distance function on $M$ and ${\Delta _H}r$ is the Laplacian of the distance function on the model space (i.e. an space form) with constant sectional curvature $H$. This theorem has many consequences in Riemannian geometry such as Meyer's theorem, Bishop-Gromov volume comparison theorem  \cite{zhu1997comparison}, Cheeger-Gromoll splitting theorem \cite{cheeger1971splitting} and their applications in topology \cite{daicomparison,zhu1997comparison} and etc.
\par There are some extensions for the Laplace operator. One of the well known extensions of Laplace operator is the weighted Laplace operator which is defined as ${\Delta _f} = \Delta  - \nabla f.\nabla $. This operator plays the same role as Laplace operator in weighted manifolds, i.e. manifolds with density ${e^{ - f}}dvo{l_g}$. Many results in comparison geometry for Laplace operator have been extended to this operator and these weighted manifolds, for example refer to \cite{brighton2013liouville,daicomparison,fang2009two,jaramillo2015fundamental,lott2003some,
wei2007comparison,wylie2017warped}. In these manifolds, the tensor $Ri{c_f} = Ric + Hessf$ plays the same role as the Ricci tensor on Riemannian manifolds. For these manifolds Wylie extended the notion of sectional curvature and got some valuable results \cite{kennard2014positive,wylie2015sectional}. The p-Laplace operator ${\Delta _p}u = div\left( {{{\left| {\nabla u} \right|}^{p - 2}}\nabla u} \right)$ is another extension of the Laplace operator and has a rich study in comparison geometry \cite{wang2012eigenvalue,wang2018gradient,wang2016lower}.
\par
Another extension of Laplace operator is the elliptic divergence type operator \linebreak ${L_A}u: = div(A\nabla u)$, where $A$ is a positive definite symmetric $(1,1)$-tensor field on a complete Riemannian manifold. A natural and major question is how to extend the results of Laplace operator and Ricci tensor to this operator. Also how can we improve the comparison results by this operator. To answer this question partially, we use a class of these operators which is called Codazzi self-adjoint operator. Codazzi operators themselves are important operators and study of them are important in their own right and have major roles in some contexts. For example they occur in natural way in Riemannian manifolds with harmonic
curvature or with harmonic Weyl tensor and many known results on such manifolds are easily obtained by properties of this tensor, fore more detail see \cite{besse2007einstein}.\\
We provide comparison results as Laplace operator and Ricci tensor to the operator $ L_A $ and an extension of Ricci tensor which is defined as $\left( {X,Y} \right) \mapsto Ric\left( {X,AY} \right)$. The results are the majors in comparison geometry such Meyer's theorem, extension of Bishop-Gromov volume comparison theorem and its results such as theorem of Yau and Calabi \cite{yau1976some} for the growth of the volume, extensions of Cheeger-Gromoll splitting theorem and its consequences in topology. Finally we get an upper estimate of the ends of manifold as same as the corresponding results for Riemannian or weighted Riemannian manifolds. As we will see in Section \ref{sec10}
 if we choose a suitable Codazzi tensor we can get an extension of the Ricci tensor which is greater the Ricci tensor, so it seems that the study of the geometry of the manifold by this tensor is more effective than the Ricci tensor also  we can study the influence of eigenvalue and eigenspaces of the operator $ A $ on the topological and geometrical properties of the manifold by the extended Ricci tensor. Also it seems, this method provide an approach to extend the results for more general cases. ( In-fact by this approach it depends on the algebraic and analytic properties of the (2,1)-tensor field  $ T^A $ which is defined in \ref{torsion}.)
 \begin{note}\label{note1}Through out the paper $ M $ is complete n-dimensional Riemannian manifold and $ A $ is positive definite self-adjoint (1,1)- Codazzi tensor field on $ M $ and bounded in the sense of Definition \ref{minmax}.
\end{note}
Explicitly our results are as follows. As usual in comparison geometry, we prove the extended mean curvature comparison theorem. We get two kinds of mean curvature comparison. The first is a comparison result about differential operator ${\Delta _{A,f}}$ and we use it to get compactness result and an extension of Cheeger-Gromoll splitting theorem. The second is about differential operator ${\L _{A,f}}$ which is used to prove the extended volume comparison theorem.
\begin{theorem}[\textbf{Extended mean curvature comparison}]\label{Meancurvaturecomparison}
Let $x_0$ a fixed point in $ M $ and $r(x) = dist(x_0,x)$ and $ H $ is some constant.
\begin{itemize}
\item[a)]
If  $\left( {n - 1} \right){\delta _n}H{\left| X \right|^2} \le Ric(X,AX)$ and $\left| {{f^A}} \right| \le K$  where $ K $ is some constant ( for $ H>0 $ assume $r \le \frac{\pi }{{4\sqrt H }}$), then along any minimal geodesic segment from $x_0$ we have,
\[{\Delta _{A,{f^A}}}r \le {\delta _n}\left( {1 + \frac{{4K}}{{{\delta _n}\left( {n - 1} \right)}}} \right){\Delta _H}r.\]
\item[b)]
If $\left( {n - 1} \right){\delta _n}H \le Ri{c_{Trace(A)}}({\partial _r},A{\partial _r})$ , $\left| {{f^A}} \right| \le K$  and $\left| {Trace(A)} \right| \le K'$ where $ K $ and ${K'}$ are some constants ( for $ H>0 $ assume $r \le \frac{\pi }{{4\sqrt H }}$), then along any minimal geodesic segment from $x_0$ we have,
\[{L_A}r \le {\delta _n}\left( {1 + \frac{{4\left( {K + K'} \right)}}{{{\delta _n}\left( {n - 1} \right)}}} \right)\left( {{\Delta _H}r} \right) + {\partial _r}.{f^A}(r),\]
\end{itemize}
where the notations ${L_{A,f}}$, ${\Delta _{A,f}}$, $Ri{c_{Trace(A)}}({\partial _r},A{\partial _r})$and ${\delta _n}$ are defined in \ref{ricci} and \ref{minmax} Also $ f^A $ is an estimates of a contraction of tensor field ${T^{\nabla A}}$ and defined in Definition \ref{Bakryidea} .
\end{theorem}
As a result of the extended mean curvature comparison for the operator ${\Delta _{A,f}}$ and with an addaptation of \cite{petersen1998integral,wei2007comparison,wu2018myers} we prove a compactness result by means of excess functions as follows.
\begin{theorem}[\textbf{Meyer's theorem}]\label{Meyer's theorem}
If $Ric(X,AX) \ge \left( {n - 1} \right){\delta _n}H{\left| X \right|^2}$ for some constant $ H>0 $ and $\left| {{f^A}} \right| \le \,K$ Then
\begin{itemize}
\item[a)]$M$ is compact and $diam(M) \le \frac{\pi }{{\sqrt H }} + \frac{{4K}}{{{\delta _n}(n - 1)\sqrt H }},$
\item[b)] $M$ has finite fundamental group.
\end{itemize}
\end{theorem}
The volume comparison theorem is one of the most important theorems in differential geometry and has many important applications, we extend the volume comparison theorem as follows.
\begin{theorem}\label{extended volume comparison}
Let $ M $ be a Riemannian manifold, $ x_0 $ be e fixed point and $ r(x):=dist(x_0,x) $, $ A $ be a self adjoint (1,1)-tensor field on it and $ R_T $ be some constant. Assume the following conditions are satisfied,
\begin{itemize}
\item[1)]For some constant $ H $ we have $Ri{c_{Trace(A)}}(X,AX) \ge \left( {n - 1} \right){\delta _n}H{\left| X \right|^2}$ ( If $ H>0 $, assume ${R_T} \le \frac{\pi }{{4\sqrt H }}$ ),
\item[2)]$\left| {{f^A}} \right| \le K$ and $\left| {Trace(A)} \right| \le K'.$
\end{itemize}
Then for $m = C\left( {{\delta _n},n,{\delta _1},K,K',H} \right) = \left[ {\frac{{{\delta _n}(n - 1) + 4(K + K')}}{{{\delta _1}}}} \right] + 2$, we have the following results,
\begin{itemize}
\item[a)]For any $0 < R \le {R_T}$, we get
\[{\left( {\frac{{vol^A(B({x_0},R))}}{{vol_H^mB(R)}}} \right)^{1/p}} - {\left( {\frac{{vol^A(B({x_0},R))}}{{vol_H^mB(r)}}} \right)^{1/p}} \le \frac{{{c_m}}}{{p{\delta _n}}}{\left\| {\left( {\delta _n^{1/p}\left( {\left| {\nabla {f^A}} \right|} \right)} \right)} \right\|_{p,R}}\int_r^R {\frac{{tsn_H^{m - 1}(t)}}{{{{\left( {vol_H^mB(t)} \right)}^{1 + 1/p}}}}dt.} \]
\item[b)]For any $0 < {r_1} \le {r_2} \le {R_1} \le {R_2} \le {R_T}$, we have the following extended volume comparison result for annular regions,
\[\begin{array}{l}
 {\left( {\frac{{vol^A(B({x_0},{r_2},{R_2}))}}{{vol_H^mB({r_2},{R_2})}}} \right)^{1/p}} - {\left( {\frac{{vol^A(B({x_0},{r_1},{R_1}))}}{{vol_H^mB({r_1},{R_1})}}} \right)^{1/p}} \\
 \,\,\,\,\,\,\,\,\,\,\,\,\,\,\,\,\,\,\,\,\,\,\,\,\,\,\,\,\,\,\,\,\,\, \le \frac{{{c_m}}}{{p{\delta _n}}}{\left\| {\left( {\delta _n^{1/p}\left( {\left| {\nabla {f^A}} \right|} \right)} \right)} \right\|_{p,R}} \times \left[ {\int_{{R_1}}^{{R_2}} {\frac{{tsn_H^{m - 1}(t)}}{{{{\left( {vol_H^mB({r_2},t)} \right)}^{1 + 1/p}}}}} dt + \int_{{r_1}}^{{r_2}} {\frac{{{R_1}sn_H^{m - 1}({R_1})}}{{{{\left( {vol_H^mB(t,{R_1})} \right)}^{1 + 1/p}}}}dt} } \right]. \\
 \end{array}\]
\item[c)] If in addition $ A $ satisfies the measure condition (\ref{invariantmeasure}) with radius measure ${r_0}>0$, then for every $0 < r < R < {R_T}$ there are some constants
\[\bar C\left( {{\delta _n},n,{\delta _1},K,K',H,p,{R_T}} \right)\]
and
\[\bar B\left( {{\delta _n},n,{\delta _1},K,K',H,p,{R_T}} \right)\]
such that, if
\[\varepsilon (p,A,{R_T}) = \frac{1}{{{{\left( {vol_{{h^A}}^A\left( {B({x_0},{R_T})} \right)} \right)}^{1/p}}}}{\left\| {\left( {\delta _n^{1/p}\left( {{r^{ - 1/p}}\left| {\nabla {{\bar f}^A}} \right|} \right)} \right)} \right\|_{p,{R_T}}} \le \bar B\left( {{\delta _n},n,{\delta _1},K,K',H,p,{R_T}} \right),\]
then
\[\frac{{vol_{{h^A}}^A(B({x_0},R))}}{{vol_{{h^A}}^A(B({x_0},r))}} \le {\left( {\frac{{1 - \bar C\varepsilon }}{{1 - 2\bar C\varepsilon }}} \right)^p}\frac{{vol_H^mB(R)}}{{vol_H^mB(r)}},\]
where
\[{\left\| {\left( {\delta _n^{1/p}\left( {{r^{ - 1/p}}\left| {\nabla {{\bar f}^A}} \right|} \right)} \right)} \right\|_{p,{R_T}}}: = {\left( {\int_{B({x_0},{R_T})} {\delta _n^{}{r^{ - 1}}{{\left| {\nabla {{\bar f}^A}} \right|}^p}dvo{l_g}} } \right)^{1/p}}.\]
and
\[vol_{{h^A}}^A(B({x_0},R)): = \int_{ B({x_0},R)} {\left\langle {{\partial _r},A{\partial _r}} \right\rangle {e^{ - {h^A}}}dvo{l_g}} ,\]
\end{itemize}
\end{theorem}
where $B(x_0,r)$ is the geodesic ball with center $x_0$ and radius $r$ and $B(x_0,{R_1},R) = B(x_0,R)\backslash B(x_0,{R_1})$ and the notations ${\delta _n},{\delta _1}$ and $vol_{{h^A}}^A$ are defined in \ref{minmax} and \ref{Dev}. Also ${{f^A}}$ is defined in Definition \ref{Bakryidea} and $Ri{c_{Trace(A)}}$ is defined as $Ri{c_{Trace(A)}}: = Ric - Hess(Trace(A))$.

One of the major and beautiful consequences of mean curvature comparison theorem, is the Cheeger-Gromoll splitting theorem, which we extend it as follows.
\begin{theorem}[\textbf{Extended Cheeger-Gromoll Splitting theorem}] \label{Cheeger} If $M$ contains a line and  $Ric(Y,AY) \ge 0$ for any vector field $ Y $,  by defining $N = {\left( {{b^ + }} \right)^{ - 1}}\left( 0 \right)$ and  ${A_N} = pro{j_{{{\left( {\nabla {b^ + }(0)} \right)}^ \bot }}} \circ A \circ pro{j_{{{\left( {\nabla {b^ + }(0)} \right)}^ \bot }}}$, we have $M = {N^{n - 1}} \times\mathbb{R} $ and $Ric(X,{A_N}X) \ge 0$, where  ${b^ + }$ is the Bussemann function associated to the ray ${\gamma _ + }(t)$ and $X \in \Gamma \left( {TN} \right)$.
\end{theorem}
The number of ends of a manifold is one of major concepts in topology and finding an upper bound for the number of ends is an important problem. We get an explicit upper bound for the number of ends of the manifold, when the tensor is nonnegative outside of a compact set.
\begin{theorem}\label{endofmanifold} Let $x_0 \in M$ be a fixed point and $ H,R>0 $ be two constants. \linebreak Assume $Ri{c_{Trace(A)}}(X,AX) \ge  - (n - 1)H{\delta _n}{\left| X \right|^2}$ in the geodesic ball $B(x_0,R)$ \linebreak and $Ri{c_{Trace(A)}}(X,AX) \ge 0$ outside the ball, also
\[\varepsilon (p,A,{R_{25R/2}}) \le \bar B\left( {{\delta _n},n,{\delta _1},K,K',H,p,{R_{25R/2}}} \right),\]
then the number of ends $ N(A,M,R) $ of the manifold $ M $ is estimated as
\[N(A,M,R) \le {\left( {\frac{{1 - \bar C\varepsilon }}{{1 - 2\bar C\varepsilon }}} \right)^p}\frac{{2m}}{{m - 1}}{\left( {\sqrt H R} \right)^{ - m}}\exp \left( {\frac{{17\left( {m - 1} \right)}}{2}\sqrt H R} \right),\]
where $p > m$ and
\begin{eqnarray*}
m:&=&  C\left( {{\delta _n},n,{\delta _1},K,K',H} \right) = \left[ {\frac{{{\delta _n}(n - 1) + 4(K + K')}}{{{\delta _1}}}} \right] + 2 \\
K: &=& \mathop {\sup }\limits_{x \in B({x_0},25R/2)} \left| {{f^A}\left( x \right)} \right|, \\
K':&=&  \mathop {\sup }\limits_{x \in B({x_0},25R/2)} \left| {Trace(A)\left( x \right)} \right|, \\
\bar C:&=&  \bar C\left( {{\delta _n},n,{\delta _1},K,K',H,p,{R_{25R/2}}} \right).
\end{eqnarray*}
\end{theorem}
The paper is organized as follows.
Section \ref{sec2}, we summarized some preliminaries. In Section \ref{sec3}, we  proved the extended mean curvature comparison Theorem \ref{Meancurvaturecomparison} and as an application, we proved the Meyer' s theorem \ref{Meyer's theorem} by means of the excess functions which it is recalled in (\ref{defexcess}). At the end of this section    some weak inequality which we used them for the extended volume comparison and estimate of the excess functions was recalled. Section \ref{sec5} is the extension of volume comparison theorem. As an application of extended volume comparison theorem we extended the theorem of Yau and Calabi \cite{yau1976some} about the growth of the volume. We generalized the Cheeger-Gromoll splitting theorem and obtained the famous  topological results of this theorem in \ref{sec7}. Section \ref{sec8} is devoted to the an estimate for the excess function and its application in topology. We proved an upper bound for the number of ends of the manifold in Section \ref{sec9}. Finally in Section \ref{sec10} as an example,  we used the extended ricci tensor for the (1,1)-self-adjoint tensor field $ A^2 $ for the study of the geometry and topology of a Riemannian hypersurface immersed isometrically in a Riemannian or Lorentzian manifold of constant sectional curvature, where $ A $ is the shape operator of the hypersurface. Although ${A^2}$ is not a Codazzi tensor, by the some restriction on $ A $, we get a similar extended radial mean curvature and get the above comparison results for the hypersurface.
\section{Preliminaries} \label{sec2}

In this section, we present some preliminaries that we use in through the  paper.
\begin{definition}
A self-adjoint operator $ A $ on a Riemannian manifold $( M,\langle , \rangle)$, is  a $ \left( {1,1} \right) $-tensor field which has the following property,
\[\forall X,Y \in \Gamma \left( {TM} \right):\left\langle {AX,Y} \right\rangle  = \left\langle {X,AY} \right\rangle. \]
\end{definition}
\begin{definition}\label{minmax}
Let $  A $ be a self-adjoint positive definite operator on a Riemannian manifold $\left( {M,\left\langle {\,,\,} \right\rangle } \right)$, we say that $ A $ is bounded if there are  constants $\delta _n ,\delta _1  > 0$ such that for any vector field $X \in \Gamma (TM)$  with $\left| X \right| = 1$, one has $\delta _1  < \left\langle {X,AX} \right\rangle  < \delta _n$.
\end{definition}

\begin{definition}\label{ricci}
Let $A $ be a self-adjoint operator on manifold $M$. We define $L_{A}, \Delta_{A}, \Delta_{A,f}, L_{A,f}, Ric_{A}$ and $Ric_{Trace A}$ as the follows:
\begin{itemize}
\item[a)]$ {L_A}(u): = div\left( {A\nabla u} \right) = \sum\nolimits_i {\left\langle {{\nabla _{{e_i}}}\left( {A\nabla u} \right),{e_i}} \right\rangle }  $,
\item[b)]$ {\Delta _A}(u): = \sum\nolimits_i {\left\langle {{\nabla _{{e_i}}}\nabla u,A{e_i}} \right\rangle } $,
\item[c)]For smooth functions $ f,u $ we define ${\Delta _{A,f}}(u) = {\Delta _A}u - \left\langle {\nabla f,\nabla u} \right\rangle ,$ also for a smooth vector field $ X $, define ${\Delta _{A,X}}u = {\Delta _A}u - \left\langle {X,\nabla u} \right\rangle.$
\item[d)]For smooth functions $ f,u $ we define ${L_{A,f}}(u): = {e^{\left( f \right)}}div\left( {{e^{ - f}}A\nabla u} \right).$
\item[e)]For smooth vector fields $ X,Y $ we define,
\[Ri{c_{Trace(A)}}(X,AY): = Ric(X,AY) - Hess\left( {Trace(A)} \right)(X,Y).\]
\item[f)]For smooth vector fields $ X,Y $ we define,
\[Ri{c_A}(X,Y): = \sum\nolimits_i {\left\langle {R(X,A{e_i}){e_i},Y} \right\rangle } ,\]
where $\left\{ {{e_i}} \right\}$ is an orthonormal basis at the point.
\end{itemize}
Note that $Ri{c_A}$ and $Ric(-,A-)$ both are extensions of Ricci tensor.
\end{definition}
By the following proposition for the distance function $r(x)$, ${\Delta _A}r$ has the same asymptotic behavior as $\Delta r$, when $r \to 0$.
\begin{prop}\label{assymptotic} Let $x_0$ be a fixed point of a Riemannian manifold $ M $ and $r(x) = dist(x_0,x)$, then $\mathop {\lim }\limits_{r \to 0} {r^2}\left( {{\Delta _A}r} \right) = 0$.
\end{prop}
\begin{proof} One knows
\[Hessr = \frac{1}{r}\left( {\left\langle\,,\,\right\rangle  - dr \otimes dr} \right) + O(1){\mkern 1mu} {\mkern 1mu} ,{\mkern 1mu} {\mkern 1mu} r \to {0^ + }.\]
 Let $\left\{ {{e_i}} \right\}$ be an orthonormal frame field with ${e_1} = \nabla r$. So by definition of \ref{ricci} we have,
\[{\Delta _A}r = \sum\nolimits_i {\left\langle {{\nabla _{{e_i}}}\nabla r,A{e_i}} \right\rangle  = } \sum\nolimits_i {Hessr\left( {{e_i},A{e_i}} \right)}  = \frac{1}{r}\left( {Trace(pro{j_{\partial _r^ \bot }} \circ {{\left. A \right|}_{\partial _r^ \bot }}} \right) + O(1){\mkern 1mu} {\mkern 1mu} {\mkern 1mu} {\mkern 1mu} ,{\mkern 1mu} {\mkern 1mu} {\mkern 1mu} {\mkern 1mu} r \to {0^ + }.\]
\end{proof}
For comparison results in geometry one needs Bochner formula. The following theorem provides the extended Bochner formula.
\begin{theorem}[\textbf{Extended Bochner formula}]\cite{alencar2015eigenvalue,gomes2016eigenvalue}\label{extended Bochner} Let $M$ be a smooth Riemannian manifold and $A$ be a self-adjoint operator on $M$, then for any smooth function $u$ on $M$,
\begin{eqnarray*}
\frac{1}{2}{L_A}({\left| {\nabla u} \right|^2})&=&  \frac{1}{2}\left\langle {\nabla {{\left| {\nabla u} \right|}^2},div(A)} \right\rangle  + Trace\left( {A \circ hes{s^2}\left( u \right)} \right) + \left\langle {\nabla u,\nabla ({\Delta _A}u)} \right\rangle  \\&&- {\Delta _{\left( {{\nabla _{\nabla u}}A} \right)}}u + Ri{c_A}(\nabla u,\nabla u),
\end{eqnarray*}
where, $Ri{c_A}$ is defined in Definition \ref{ricci}.
\end{theorem}
If $ A $ is a self-adjoint operator, then we define $ T^A $ as follows.
\begin{definition}\label{torsion} Let $ A $ be a $ (1,1)$-tensor field on $ M $. Define $ T^A $ as,
  \[{T^A}(X,Y): = \left( {{\nabla _X}A} \right)Y - \left( {{\nabla _Y}A} \right)X.\]
It is clear $T$ is a (2,1) tensor field and when $ T^A=0 $, then $ A $ is a Codazzi tensor, that is,   $(\nabla _{X}A)\langle Y,Z\rangle=(\nabla _{Y}A)\langle X,Z\rangle$.
\end{definition}
\begin{example}
If $B$ is the shape operator of a hypersurface ${\Sigma ^n} \subset {M^{n + 1}}$ then \[{T^B}(Y,X) = {\left( {\bar R(Y,X)N} \right)^T},\]
where ${\bar R}$ is the curvature tensor of $ M $ and $ N $ is the unit normal vector field on ${\Sigma ^n} \subset {M^{n + 1}}$.
\end{example}
\begin{lem}\label{ttorsion} Let $ B $ be a (1,1)-self-adjoint tensor field, then,
\[\left\langle {X,{T^B}(Y,Z)} \right\rangle  = \left\langle {{T^B}(Y,X),Z} \right\rangle  + \left\langle {{T^B}(X,Z),Y} \right\rangle .\]
\end{lem}
\begin{proof} By computation, we have,
\begin{eqnarray*}
 \left\langle {X,{T^B}(Y,Z)} \right\rangle&=& \left\langle {X,\left( {{\nabla _Y}B} \right)Z - \left( {{\nabla _Z}B} \right)Y} \right\rangle
\\&=&\left\langle {\left( {{\nabla _Y}B} \right)X - \left( {{\nabla _X}B} \right)Y,Z} \right\rangle  + \left\langle {\left( {{\nabla _X}B} \right)Y,Z} \right\rangle  - \left\langle {X,\left( {{\nabla _Z}B} \right)Y} \right\rangle
\\&=&\left\langle {\left( {{\nabla _Y}B} \right)X - \left( {{\nabla _X}B} \right)Y,Z} \right\rangle  + \left\langle {Y,\left( {{\nabla _X}B} \right)Z - \left( {{\nabla _Z}B} \right)X} \right\rangle
\\&=&\left\langle {{T^B}(Y,X),Z} \right\rangle  + \left\langle {{T^B}(X,Z),Y} \right\rangle .
\end{eqnarray*}
\end{proof}
We compute the second covariant derivation of the operator $ B $. The result is applied to simplify the extended Bochner formula.
\begin{lem}\label{laplaceA1} Let $ B $ be a (1,1)-self-adjoint tensor field on the manifold $ M $ and $ X,Y,Z $  are vector fields on $ M $, then
\begin{itemize}
\item[a)]$\left( {{\nabla ^2}B} \right)\left( {X,Y,Z} \right) = \left( {{\nabla ^2}B} \right)\left( {X,Z,Y} \right) + R(Z,Y)\left( {BX} \right) - B\left( {R(Z,Y)X} \right),$
\item[b)]$\left( {{\nabla ^2}B} \right)\left( {X,Y,Z} \right) - \left( {{\nabla ^2}B} \right)\left( {Y,X,Z} \right) = \left( {{\nabla _Z}T^B} \right)(X,Y).$
\end{itemize}
\end{lem}
\begin{proof}
For part (a) we have,
\begin{eqnarray*}
{\nabla ^2}B(X,Y,Z) &=&\left( {\nabla \left( {\nabla B} \right)} \right)(X,Y,Z) = \left( {{\nabla _Z}\left( {\nabla B} \right)} \right)(X,Y)
\\&=& {\nabla _Z}\left( {\left( {\nabla B} \right)(X,Y)} \right) - \left( {\nabla B} \right)({\nabla _Z}X,Y) - \left( {\nabla B} \right)(X,{\nabla _Z}Y)
\\&=&{\nabla _Z}\left( {\left( {{\nabla _Y}B} \right)X} \right) - \left( {\nabla B} \right)({\nabla _Z}X,Y) - \left( {{\nabla _{{\nabla _Z}Y}}B} \right)(X)
\\&=& \left( {{\nabla _Z}\left( {{\nabla _Y}B} \right)} \right)X + \left( {{\nabla _Y}B} \right)\left( {{\nabla _Z}X} \right) - \left( {{\nabla _Y}B} \right)({\nabla _Z}X) - \left( {{\nabla _{{\nabla _Z}Y}}B} \right)(X)
\\&=&\left( {{\nabla _Z}\left( {{\nabla _Y}B} \right)} \right)X - \left( {{\nabla _{{\nabla _Z}Y}}B} \right)X
\end{eqnarray*}
Similarly,
\[{\nabla ^2}B(X,Z,Y) = \left( {{\nabla _Y}\left( {{\nabla _Z}B} \right)} \right)X - \left( {{\nabla _{{\nabla _Y}Z}}B} \right)X.\]
 Thus
\begin{eqnarray*}
{\nabla ^2}B(X,Y,Z) - {\nabla ^2}B(X,Z,Y){\rm{ }}&=&\left( {{\nabla _Z}{\nabla _Y}B} \right)X - \left( {{\nabla _Y}\left( {{\nabla _Z}B} \right)} \right)X - \left( {{\nabla _{\left[ {Z,Y} \right]}}B} \right)X
\\&=&\left( {R(Z,Y)B} \right)X = R(Z,Y)\left( {BX} \right) - B\left( {\left( {R(Z,Y)X} \right)} \right).
\end{eqnarray*}
For part (b), by definition of $ T $, we have
\begin{eqnarray*}
{\nabla ^2}B(X,Y,Z)&=&\left( {{\nabla _Z}\left( {\nabla B} \right)} \right)\left( {X,Y} \right)
\\&=&{\nabla _Z}\left( {\left( {\nabla B} \right)\left( {X,Y} \right)} \right) - \left( {\nabla B} \right)\left( {{\nabla _Z}X,Y} \right) - \left( {\nabla B} \right)\left( {X,{\nabla _Z}Y} \right)
\\&=&{\nabla _Z}\left( {\left( {\nabla B} \right)\left( {Y,X} \right) + {T^B}\left( {Y,X} \right)} \right) - \left( {\nabla B} \right)\left( {{\nabla _Z}X,Y} \right) - \left( {\nabla B} \right)\left( {X,{\nabla _Z}Y} \right)
\\&=&{\nabla _Z}\left( {\left( {\nabla B} \right)\left( {Y,X} \right)} \right) + {\nabla _Z}\left( {{T^B}\left( {Y,X} \right)} \right) - \left( {\nabla B} \right)\left( {{\nabla _Z}X,Y} \right) - \left( {\nabla B} \right)\left( {X,{\nabla _Z}Y} \right)
\\&=&\left( {{\nabla _Z}\left( {\nabla B} \right)\left( {Y,X} \right)} \right) + \left( {\nabla B} \right)\left( {{\nabla _Z}Y,X} \right) + \left( {\nabla B} \right)\left( {Y,{\nabla _Z}X} \right)
\\&&+ {\nabla _Z}\left( {{T^B}\left( {Y,X} \right)} \right) - \left( {\nabla B} \right)\left( {{\nabla _Z}X,Y} \right) - \left( {\nabla B} \right)\left( {X,{\nabla _Z}Y} \right)
\\&=& \left( {{\nabla _Z}\left( {\nabla B} \right)\left( {Y,X} \right)} \right) + {\nabla _Z}\left( {{T^B}\left( {Y,X} \right)} \right) - {T^B}\left( {{\nabla _Z}Y,X} \right) - {T^B}\left( {Y,{\nabla _Z}X} \right)
\\&=&\left( {{\nabla _Z}\left( {\nabla B} \right)\left( {Y,X} \right)} \right) + \left( {{\nabla _Z}{T^B}} \right)\left( {X,Y} \right).
\end{eqnarray*}
\end{proof}
\begin{lem}\label{laplaceA} Let $ B $ be a $ (1,1)-$symmetric tensor field and ${\nabla ^*}{T^B} = 0$, then
\[\left\langle {\left( {\Delta B} \right)X,X} \right\rangle  = \left\langle {{\nabla _X}div(B),X} \right\rangle  - Ri{c_B}(X,X) + Ric(X,BX),\]
where $\nabla ^*$ is adjoint of $\nabla$.
\end{lem}
\begin{proof} For simplicity let $\left\{ {{e_i}} \right\}$ be an orthonormal local frame field with ${\nabla _{{e_i}}}{e_j} = 0$ at the computing point. By computation and Lemma \ref{laplaceA1} we have,
\begin{eqnarray*}
\left\langle {\left( {\Delta B} \right)X,X} \right\rangle&=&\sum\nolimits_i {\left\langle {\left( {{\nabla _{{e_i}}}{\nabla _{{e_i}}}B} \right)X,X} \right\rangle }  = \sum\nolimits_i {\left\langle {{\nabla ^2}B(X,{e_i},{e_i}),X} \right\rangle }
\\&=& \sum\nolimits_i {\left\langle {{\nabla ^2}B({e_i},X,{e_i}),X} \right\rangle }  + \sum\nolimits_i {\left\langle {\left( {{\nabla _{{e_i}}}{T^B}} \right)(X,{e_i}),X} \right\rangle }
\\&=& \sum\nolimits_i {\left\langle {{\nabla ^2}B({e_i},X,{e_i}),X} \right\rangle } .
\end{eqnarray*}
 So by Lemma \ref{laplaceA1}, part (a) we have
\begin{eqnarray*}
\left\langle {\left( {\Delta B} \right)X,X} \right\rangle &=& \sum\nolimits_i {\left\langle {{\nabla ^2}B({e_i},X,{e_i}),X} \right\rangle }
\\&=&\sum\nolimits_i {\left\langle {{\nabla ^2}B({e_i},{e_i},X) + R({e_i},X)\left( {B{e_i}} \right) -B\left( {\left( {R({e_i},X){e_i}} \right)} \right),X} \right\rangle }
\\&=& \left\langle {\left( {{\nabla _X}divB} \right),X} \right\rangle  - Ri{c_B}\left( {X,X} \right) + Ric\left( {X,BX} \right).
\end{eqnarray*}
\end{proof}
The extended Bochner formula \ref{extended Bochner} is very complicated. The complication of the formula is about the existence of parameter ${\Delta _{{\nabla _u}B}}u $. Howerer, when ${{\Delta _{{\nabla _u}B}}u \le 0}$, the Bochner formula get the simple Riccati inequality. In the following proposition, we get some parameters which seems are suitable to estimate ${{\Delta _{{\nabla _u}B}}u}$. These parameters show the affection of parallelness and values of eigenfunction on the value of ${{\Delta _{{\nabla _u}B}}u}$. The values of these parameter depend on analytic and algebraic properties of the tensor $ B $.
\begin{prop}\label{lcomplicated} Let $ B $ be a (1,1)-self-adjoint tensor field on the manifold $ M $ then
\begin{eqnarray*}
{\Delta _{\left( {{\nabla _{{\nabla _u}}}B} \right)}}u&=&\nabla u.\nabla u.Trace(B) - \left\langle {\nabla u,\left( {\Delta B} \right)\nabla u} \right\rangle  + \sum\nolimits_i {\left\langle {{T^{\left( {{\nabla _{\nabla u}}B} \right)}}({e_i},\nabla u),{e_i}} \right\rangle }
\\&&+ \sum\nolimits_i {{e_i}.\left\langle {\nabla u,{T^B}({e_i},\nabla u)} \right\rangle .}
\end{eqnarray*}
\end{prop}
\begin{proof}
Let $ B $ be a  $(1,1)$-tensor field, then
\begin{eqnarray*}
{\Delta _{\left( {{\nabla _{{\nabla _u}}}B} \right)}}u&=&  \sum\nolimits_i {\left\langle {{\nabla _{{e_i}}}\nabla u,\left( {{\nabla _{{\nabla _u}}}B} \right){e_i}} \right\rangle }\\
&=& \sum\nolimits_i {\left\langle {{\nabla _{{e_i}}}\nabla u,\left( {{\nabla _{{e_i}}}B} \right)\nabla u} \right\rangle }  + \sum\nolimits_i {\left\langle {{\nabla _{{e_i}}}\nabla u,{T^B}(\nabla u,{e_i})} \right\rangle }
\\&=&\sum\nolimits_i {{e_i}.\left\langle {\nabla u,\left( {{\nabla _{{e_i}}}B} \right)\nabla u} \right\rangle  - \sum\nolimits_i {\left\langle {\nabla u,\left( {\nabla _{{e_i}}^2B} \right)\nabla u} \right\rangle } }
\\&&- \sum\nolimits_i {\left\langle {{\nabla _{{e_i}}}\nabla u,\left( {{\nabla _{{e_i}}}B} \right)\nabla u} \right\rangle }  + \sum\nolimits_i {\left\langle {{\nabla _{{e_i}}}\nabla u,{T^B}(\nabla u,{e_i})} \right\rangle }
\\&=&\sum\nolimits_i {{e_i}.\left\langle {\nabla u,\left( {{\nabla _{{e_i}}}B} \right)\nabla u} \right\rangle }  - \left\langle {\nabla u,\left( {\Delta B} \right)\nabla u} \right\rangle  - {\Delta _{\left( {{\nabla _{{\nabla _u}}}B} \right)}}u
\\&& - \sum\nolimits_i {\left\langle {{\nabla _{{e_i}}}\nabla u,{T^B}({e_i},\nabla u)} \right\rangle }  + \sum\nolimits_i {\left\langle {{\nabla _{{e_i}}}\nabla u,{T^B}(\nabla u,{e_i})} \right\rangle }
\\&=& \sum\nolimits_i {{e_i}.\left\langle {\nabla u,\left( {{\nabla _{{e_i}}}B} \right)\nabla u} \right\rangle }  - \left\langle {\nabla u,\left( {\Delta B} \right)\nabla u} \right\rangle  - {\Delta _{\left( {{\nabla _{{\nabla _u}}}B} \right)}}u
\\&& + 2\sum\nolimits_i {\left\langle {{\nabla _{{e_i}}}\nabla u,{T^B}(\nabla u,{e_i})} \right\rangle }
\\&=&\sum\nolimits_i {{e_i}.\left\langle {\nabla u,\left( {{\nabla _{\nabla u}}B} \right){e_i}} \right\rangle }  + \sum\nolimits_i {{e_i}.\left\langle {\nabla u,{T^B}({e_i},\nabla u)} \right\rangle }  - \left\langle {\nabla u,\left( {\Delta B} \right)\nabla u} \right\rangle
\\&&- {\Delta _{\left( {{\nabla _{{\nabla _u}}}B} \right)}}u + 2\sum\nolimits_i {\left\langle {{\nabla _{{e_i}}}\nabla u,{T^B}(\nabla u,{e_i})} \right\rangle }
\\&=&{\Delta _{\left( {{\nabla _{{\nabla _u}}}B} \right)}}u + \sum\nolimits_i {\left\langle {\nabla u,\left( {{\nabla _{{e_i}}}{\nabla _{\nabla u}}B} \right){e_i}} \right\rangle }  + \sum\nolimits_i {\left\langle {{\nabla _{{e_i}}}\nabla u,{T^B}({e_i},\nabla u)} \right\rangle }
\\&&+ \sum\nolimits_i {\left\langle {\nabla u,{\nabla _{{e_i}}}\left( {{T^B}({e_i},\nabla u)} \right)} \right\rangle }  - \left\langle {\nabla u,\left( {\Delta B} \right)\nabla u} \right\rangle  - {\Delta _{\left( {{\nabla _{{\nabla _u}}}B} \right)}}u
\\&&+ 2\sum\nolimits_i {\left\langle {{\nabla _{{e_i}}}\nabla u,{T^B}(\nabla u,{e_i})} \right\rangle }
\end{eqnarray*}
In other words,
\[{\Delta _{\left( {{\nabla _{{\nabla _u}}}B} \right)}}u = \left\langle {\nabla u,div\left( {{\nabla _{\nabla u}}B} \right)} \right\rangle  + \sum\nolimits_i {{e_i}.\left\langle {\nabla u,{T^B}({e_i},\nabla u)} \right\rangle }  - \left\langle {\nabla u,\left( {\Delta B} \right)\nabla u} \right\rangle .\]
But,
\begin{eqnarray*}
\left\langle {\nabla u,div\left( {{\nabla _{\nabla u}}B} \right)} \right\rangle&=&\sum\nolimits_i {\left\langle {\nabla u,\left( {{\nabla _{{e_i}}}\left( {{\nabla _{\nabla u}}B} \right)} \right){e_i}} \right\rangle }  = \sum\nolimits_i {\left\langle {\left( {{\nabla _{{e_i}}}\left( {{\nabla _{\nabla u}}B} \right)} \right)\nabla u,{e_i}} \right\rangle }
\\&=&\sum\nolimits_i {\left\langle {\left( {{\nabla _{\nabla u}}\left( {{\nabla _{\nabla u}}B} \right)} \right){e_i} + {T^{\left( {{\nabla _{\nabla u}}B} \right)}}({e_i},\nabla u),{e_i}} \right\rangle }
\\&=&\nabla u.\nabla u.Trace(B) + \sum\nolimits_i {\left\langle {{T^{\left( {{\nabla _{\nabla u}}B} \right)}}({e_i},\nabla u),{e_i}} \right\rangle .}
\end{eqnarray*}
So,
\begin{eqnarray*}
{\Delta _{\left( {{\nabla _{{\nabla _u}}}B} \right)}}u&=&\nabla u.\nabla u.Trace(B) - \left\langle {\nabla u,\left( {\Delta B} \right)\nabla u} \right\rangle  + \sum\nolimits_i {\left\langle {{T^{\left( {{\nabla _{\nabla u}}B} \right)}}({e_i},\nabla u),{e_i}} \right\rangle }
\\&&+\sum\nolimits_i {{e_i}.\left\langle {\nabla u,{T^B}({e_i},\nabla u)} \right\rangle .}
\end{eqnarray*}
\end{proof}
So, the extended Bochner formula \ref{extended Bochner} can be written as follows.
\begin{prop}\label{extendedbochner2} Let $ B $ a (1,1)-self-adjoint  tensor field on a manifold $ M $ with ${\nabla ^*}{T^B} = 0$  and $ u $ be a smooth function, then
\begin{eqnarray*}
\frac{1}{2}{L_B}({{\left| {\nabla u} \right|}^2}) &=& \frac{1}{2}\left\langle {\nabla {{\left| {\nabla u} \right|}^2},div(B)} \right\rangle  + Trace\left( {B \circ hes{s^2}\left( u \right)} \right) + \left\langle {\nabla u,\nabla ({\Delta _B}u)} \right\rangle\\
&&- \nabla u.\nabla u.Trace(B) + \left\langle {{\nabla _{\nabla u}}div(B),\nabla u} \right\rangle  + Ric(\nabla u,B\nabla u)
\\&&- \sum\nolimits_i {\left\langle {{T^{\left( {{\nabla _{\nabla u}}B} \right)}}({e_i},\nabla u),{e_i}} \right\rangle }  - \sum\nolimits_i {{e_i}.\left\langle {\nabla u,{T^B}({e_i},\nabla u)} \right\rangle .}
\end{eqnarray*}
\end{prop}
\begin{proof} The proof is clear by Proposition \ref{lcomplicated} and Theorem \ref{extended Bochner} and Lemma \ref{laplaceA}.
\end{proof}
\section{extended mean curvature comparison}\label{sec3}
In this section, we prove two version of extended mean curvature comparison theorem, when $ A $ is a $(1,1)$-self-adjoint Codazzi tensor. The first one is about ${\Delta _{A,f}}$, which is used for extended Meyer's theorem, extended Cheeger-Gromoll splitting theorem and and estimating of the excess functions. The second is about ${L_{A,f}}$ which is used to prove extended volume comparison and its topological results. For the first one we used the tensor $Ric({\partial _r},A{\partial _r})$ and for the second one we apply the tensor $Ri{c_{Trace(A)}}({\partial _r},A{\partial _r}) = Ric({\partial _r},A{\partial _r}) + {\partial _r}.\left\langle {divA,{\partial _r}} \right\rangle $. Let  $ x_0 $ be a fixed point of $  M $, we define $ r(x) = dist(x_0, x) $, then $r(x)$ is smooth on $x \in M\backslash cut(x_0)$  and $\left| {\nabla r} \right| = 1$. For simplicity, we denote $\nabla r$ by ${\partial _r}$. So by Theorem \ref{extendedbochner2} we get the following result.
\begin{theorem}\label{bochner3} Let $ A $ be a (1,1)-self-adjoint Codazzi tensor and $ r(x) = dist(x_0, x) $ be the distance function from a fixed point $x_0 $, then on $M\backslash cut(x_0)$ we have,
\[0 = Trace\left( {A \circ hes{s^2}\left( r \right)} \right) + {\partial _r}.({\Delta _A}r) + Ric({\partial _r},A{\partial _r}) - \sum\nolimits_i {\left\langle {{T^{\left( {{\nabla _{{\partial _r}}}A} \right)}}({e_i},{\partial _r}),{e_i}} \right\rangle } .\]
\end{theorem}
To get the extended mean curvature comparison, we need to approximate the parameter $\sum\nolimits_i {\left\langle {{T^{\left( {{\nabla _{{\partial _r}}}A} \right)}}({\partial _r},{e_i}),{e_i}} \right\rangle } $. When this parameter is negative, a similar Riccati inequality follows, but when this parameter is positive it is more complicated. To solve this problem we estimate this parameter by $Hessf({\partial _r},{\partial _r})$ and adapt the approach of \cite{wei2007comparison}. So at first, we define the following.
\begin{definition}\label{Bakryidea} We define the continuous function ${F^A}$ as follows,
\[{F^A}(x): = \mathop {\max }\limits_{X \in {T_x}M,\left| X \right| = 1} \sum\nolimits_i {\left\langle {{T^{\left( {{\nabla _X}A} \right)}}({e_i},X),{e_i}} \right\rangle } (x).\]
Also we define the function ${f^A}$ as a smooth function which satisfy the following,
\begin{itemize}
\item[(*)]${F^A}{g_M} \le Hess\left( {{f^A}} \right)$, in the sense of quadratic forms, where ${g_M}$ is the metric tensor of the manifold $ M $.
\end{itemize}
\end{definition}
In the following we give some examples for the ${f^A}$.
\begin{example}For the following $ f^A=0 $.
\begin{itemize}
\item[a)]When ${\Sigma ^n} \subset {M^{n + 1}}(c)$ be a totally ubilical hypersurfaces, then ${\Delta _{{\nabla _{\nabla u}}A}}u = 0$, so $ F^A=0 $ thus $ f^A=0 $.
\item[b)]For a Codazzi tensor $ A $, if ${\nabla ^2}A = 0$ then $ f^A=0 $.
\end{itemize}
 \end{example}
\begin{example} If for any vector field $X,T,Z \in \Gamma (TM)$ with $\left| X \right| = \left| Y \right| = \left| Z \right| = 1$ we have$\left\| {\left( {{\nabla _X}{\nabla _Y}A} \right)Z} \right\|(x) \le K(x)$, then
${F^A}(x) \le 2nK(x)$. Let $ f $ be a radial function, we try to find some conditions to have $2nK(x) \le Hessf(X,X)$, where $\left| X \right| = 1$. Note that $Hessf(X,X) = f''(r) + f'(r)Hessr(X,X)$, let $f' > 0$, to find $ f $ its sufficient to have
\[\mathop {\sup }\limits_{x \in B(r)} 2nK(x) \le \left( {f'' + \frac{{h'}}{h}f'} \right),\]
where $ h $ is the solution of differential equation,
\[\left\{ {\begin{array}{*{20}{c}}
   {h'' - Gh = 0},  \\
   {h(0) = 0\,\,,\,h'(0) = 1},  \\
\end{array}} \right.\]
where ${\sec _{rad}} \le  - G$, see \cite{pigola2008vanishing}.
 \end{example}
Now we present the proof of Theorem \ref{Meancurvaturecomparison}.
\begin{proof}[\textbf{Proof of Theorem \ref{Meancurvaturecomparison}}]
For the first part, by assumption $  A $ is positive semi-definite, so for any smooth function $ u$
we have
\[Trace\left( {A \circ hes{s^2}\left( u \right)} \right) \ge {\textstyle{{{{({\Delta _A}u)}^2}} \over {(TraceA)}}},\]
and by restriction on $ A $ in Definition \ref{minmax}, we have  $\frac{1}{{\left( {n - 1} \right){\delta _n}}} \le \frac{1}{{Trace(A)}} \le \frac{1}{{\left( {n - 1} \right){\delta _1}}}$ so we get the corresponding Ricatti inequality.
\begin{equation}
0 \ge \frac{{{{\left( {{\Delta _A}r} \right)}^2}}}{{\left( {n - 1} \right){\delta _n}}} + {\partial _r}.({\Delta _A}r) + Ric({\partial _r},A{\partial _r}) - {\partial _r}.{\partial _r}.{f^A}.
\label{bocnerin}
\end{equation}
Let $\gamma (t)$ be a minimal geodesic from the point $x_0$. So on this geodesic,
\[0 \ge \frac{{{{\left( {{\Delta _A}r} \right)}^2}}}{{\left( {n - 1} \right){\delta _n}}} + ({\Delta _A}r)' + Ric\left( {\gamma '(t),A\gamma '(t)} \right) - {\left( {{f^A}(t)} \right)^{\prime \prime }}.\]
Let ${M_H^n}$ be the model space with constant sectional curvature $ H $, then
\[\frac{{{{({\Delta _H}r)}^2}}}{{n - 1}} + ({\Delta _H}r)' + (n - 1)H = 0.\]
We know that ${\Delta _H}r = \left( {n - 1} \right)\frac{{s{n'_H}(r)}}{{s{n_H}(r)}}$, where
\[s{n_H}(r) = \left\{ {\begin{array}{*{20}{c}}
   {\frac{1}{{\sqrt H }}\sin (\sqrt H r)\,\,\,\,\,\,\,\,\,\,\,\,\,\,\,\,H > 0,}  \\
   {r\,\,\,\,\,\,\,\,\,\,\,\,\,\,\,\,\,\,\,\,\,\,\,\,\,\,\,\,\,\,\,\,\,\,\,\,\,\,\,\,\,\,\,\,\,\,\,\,\,\,H = 0,}  \\
   {\frac{1}{{\sqrt { - H} }}\sinh (\sqrt { - H} r)\,\,\,\,\,\,H < 0.}  \\
\end{array}} \right.\]
By assumption, we have $\left( {n - 1} \right){\delta _n}H \le Ric({\partial _r},{A\partial _r})$. So,
\begin{equation}
{\left( {\frac{{{\Delta _A}r}}{{{\delta _n}}} - {\Delta _H}r} \right)^\prime } \le  - \left( {\frac{{{{({\Delta _A}r)}^2}}}{{\left( {n - 1} \right)\delta _n^2}} - \frac{{{{({\Delta _H}r)}^2}}}{{n - 1}}} \right) + \frac{1}{{{\delta _n}}}{\left( {{f^A}(t)} \right)^{\prime \prime }}.
\label{exm1}
\end{equation}
Using ( \ref{exm1} ) and computation give that
\begin{eqnarray*}
{{\left( {sn_H^2(r)\left( {\frac{{{\Delta _A}r}}{{{\delta _n}}} - {\Delta _H}r} \right)} \right)}^\prime }&=& 2s{{n'}_H}(r)s{n_H}(r)\left( {\frac{{{\Delta _A}r}}{{{\delta _n}}} - {\Delta _H}r} \right) + sn_H^2(r){{\left( {\frac{{{\Delta _A}r}}{{{\delta _n}}} - {\Delta _H}r} \right)}^\prime }
\\&\le&\frac{{ 2sn_H^2(r)}}{{( {n - 1} )}}( {{\Delta _H}r} )( {\frac{{{\Delta _A}r}}{{{\delta _n}}} - {\Delta _H}r} ) - sn_H^2(r)( {\frac{{{{({\Delta _A}r)}^2}}}{{( {n - 1} )\delta _n^2}} - \frac{{{{({\Delta _H}r)}^2}}}{{n - 1}}} )
\\&&+ \frac{{sn_H^2(r)}}{{{\delta _n}}}{{\left( {{f^A}(r)} \right)}^{\prime \prime }}
\\&=& \frac{{sn_H^2(r)}}{{\left( {n - 1} \right)}}\left( 2{\frac{{\left( {{\Delta _H}r} \right)\left( {{\Delta _A}r} \right)}}{{{\delta _n}}} - 2{{\left( {{\Delta _H}r} \right)}^2} - \frac{{{{({\Delta _A}r)}^2}}}{{\delta _n^2}} + {{\left( {{\Delta _H}r} \right)}^2}} \right)
\\&&+\frac{{sn_H^2(r)}}{{{\delta _n}}}{\left( {{f^A}(r)} \right)^{\prime \prime }}
\\&=& - \frac{{sn_H^2(r)}}{{\left( {n - 1} \right)}}{{\left( {\frac{{{\Delta _A}r}}{{{\delta _n}}} - {\Delta _H}r} \right)}^2} + \frac{{sn_H^2(r)}}{{{\delta _n}}}{{\left( {{f^A}(r)} \right)}^{\prime \prime }}
\\& \le& \frac{{sn_H^2(r)}}{{{\delta _n}}}{{\left( {{f^A}(r)} \right)}^{\prime \prime }}.
\end{eqnarray*}
By Proposition \ref{assymptotic} $\mathop {\lim }\limits_{r \to 0} sn_H^2(r)\left( {\frac{{{\Delta _A}r}}{{{\delta _n}}} - {\Delta _H}r} \right) = 0$. So by integration with respect to $ r $, we get,
\begin{eqnarray*}
\frac{1}{{{\delta _n}}}sn_H^2(r)\left( {{\Delta _A}r} \right)&\le&sn_H^2(r)\left( {{\Delta _H}r} \right) + \frac{1}{{{\delta _n}}}\int_0^r {sn_H^2(t){{\left( {{f^A}(t)} \right)}^{\prime \prime }}dt}\\
&=&sn_H^2(r)\left( {{\Delta _H}r} \right) + \frac{1}{{{\delta _n}}}sn_H^2(r){{\left( {{f^A}(r)} \right)}^\prime } - \frac{1}{{{\delta _n}}}\int_0^r {{{\left( {sn_H^2(t)} \right)}^\prime }{{\left( {{f^A}(t)} \right)}^\prime }dt} .
\end{eqnarray*}
So by definition of ${\Delta _{A,{f^A}}}$ in \ref{ricci} one has,
\[\frac{1}{{{\delta _n}}}sn_H^2(r)\left( {{\Delta _{A,{f^A}}}(r)} \right) \le sn_H^2(r)\left( {{\Delta _H}r} \right) - \frac{1}{{{\delta _n}}}\int_0^r {{{\left( {sn_H^2(t)} \right)}^\prime }{{\left( {{f^A}(t)} \right)}^\prime }dt} .\]
By applying integration by parts again,
\[\frac{1}{{{\delta _n}}}sn_H^2(r)\left( {{\Delta _{A,{f^A}}}(r)} \right) \le sn_H^2(r)\left( {{\Delta _H}r} \right) - \frac{1}{{{\delta _n}}}{f^A}(r){\left( {sn_H^2(r)} \right)^\prime } + \frac{1}{{{\delta _n}}}\int_0^r {{{\left( {sn_H^2(t)} \right)}^{\prime \prime }}{f^A}(t)dt.} \]
When $ H>0 $ by assumption, we have $r \le \frac{\pi }{{4\sqrt H }}$, so ${\left( {sn_H^2(t)} \right)^{\prime \prime }} \ge 0$, thus
\[\frac{1}{{{\delta _n}}}sn_H^2(r)\left( {{\Delta _{A,{f^A}}}(r)} \right) \le sn_H^2(r)\left( {{\Delta _H}r} \right) + \frac{{2K}}{{{\delta _n}}}{\left( {sn_H^2(r)} \right)^\prime }.\]
We know,
\[{\left( {sn_H^2(r)} \right)^\prime } = 2{\left( {sn_H^{}(r)} \right)^\prime }sn_H^{}(r) = \frac{2}{{n - 1}}\left( {{\Delta _H}r} \right)\left( {sn_H^2(r)} \right),\]
so,
\[\left( {{\Delta _{A,{f^A}}}(r)} \right) \le {\delta _n}\left( {1 + \frac{{4K}}{{{\delta _n}\left( {n - 1} \right)}}} \right)\left( {{\Delta _H}r} \right).\]
For the second part, from the condition on $Ri{c_{Trace(A)}}\left( {{\partial _r},A{\partial _r}} \right)$ we have,
\[{\left( {\frac{{{\Delta _A}}}{{{\delta _n}}} - {\Delta _H}r} \right)^\prime } \le  - \left( {\frac{{{{({\Delta _A}r)}^2}}}{{\left( {n - 1} \right)\delta _n^2}} - \frac{{{{({\Delta _H}r)}^2}}}{{n - 1}}} \right) + \frac{1}{{{\delta _n}}}{\left( {{f^A}(t) - Trace(A)} \right)^{\prime \prime }}.\]
With similar computation,
\[{\left( {sn_H^2(r)\left( {\frac{{{\Delta _A}r}}{{{\delta _n}}} - {\Delta _H}r} \right)} \right)^\prime } \le \frac{{sn_H^2(r)}}{{{\delta _n}}}{\left( {{f^A}(t) - Trace(A)} \right)^{\prime \prime }}.\]
Thus,
\begin{eqnarray*}
\frac{1}{{{\delta _n}}}sn_H^2(r)\left( {{\Delta _A}r} \right)&\le&sn_H^2(r)\left( {{\Delta _H}r} \right) + \frac{1}{{{\delta _n}}}\int_0^r {sn_H^2(t){{\left( {{f^A}(t) - Trace(A)} \right)}^{\prime \prime }}dt}
\\&=& sn_H^2(r)\left( {{\Delta _H}r} \right) + \frac{1}{{{\delta _n}}}sn_H^2(r){{\left( {{f^A}(r) - Trace(A)} \right)}^\prime }
\\&&-\frac{1}{{{\delta _n}}}\int_0^r {{{\left( {sn_H^2(t)} \right)}^\prime }{{\left( {{f^A}(t) - Trace(A)} \right)}^\prime }dt}.
\end{eqnarray*}
Note that $ A $ is a Codazzi tensor, so $divA = \nabla Trace(A)$, so
\begin{eqnarray*}
\frac{1}{{{\delta _n}}}sn_H^2(r)\left( {\left( {{\Delta _A}r} \right) + \left\langle {divA,{\partial _r}} \right\rangle  - {\partial _r}.{f^A}(r)} \right)&\le&sn_H^2(r)\left( {{\Delta _H}r} \right)
\\&& - \frac{1}{{{\delta _n}}}\int_0^r {{{\left( {sn_H^2(t)} \right)}^\prime }{{\left( {{f^A}(t) - Trace(A)} \right)}^\prime }dt} ,
\end{eqnarray*}
and ${L_A}r = \left\langle {divA,{\partial _r}} \right\rangle  + {\Delta _A}r$, thus,
\[\frac{1}{{{\delta _n}}}sn_H^2(r)\left( {{L_A}r - {\partial _r}.{f^A}(r)} \right) \le sn_H^2(r)\left( {{\Delta _H}r} \right) - \frac{1}{{{\delta _n}}}\int_0^r {{{\left( {sn_H^2(t)} \right)}^\prime }{{\left( {{f^A}(t) - Trace(A)} \right)}^\prime }dt} .\]
Similarly,
\begin{equation}\label{meancomp1}
{L_A}r \le {\delta _n}\left( {1 + \frac{{4\left( {K + K'} \right)}}{{{\delta _n}\left( {n - 1} \right)}}} \right)\left( {{\Delta _H}r} \right) + {\partial _r}.{f^A}(r).
\end{equation}
\end{proof}
\begin{remark}\label{largercomparison} When $ H>0 $, for $\frac{\pi }{{4\sqrt H }} \le r \le \frac{\pi }{{2\sqrt H }}$ we have
\[\int_0^r {{{\left( {sn_H^2(t)} \right)}^{\prime \prime }}{f^A}(t)dt}  \le K\left( {\int_0^{\frac{\pi }{{4\sqrt H }}} {{{\left( {sn_H^2(t)} \right)}^{\prime \prime }}dt} \, - \int_{\frac{\pi }{{4\sqrt H }}}^r {{{\left( {sn_H^2(t)} \right)}^{\prime \prime }}dt} \,\,\,} \right) = K\left( {\frac{2}{{\sqrt H }}\, - s{n_H}(2r)} \right).\]
So, when $\frac{\pi }{{4\sqrt H }} \le r \le \frac{\pi }{{2\sqrt H }}$,
\[\left( {{\Delta _{A,{f^A}}}(r)} \right) \le {\delta _n}\left( {1 + \frac{1}{{{\delta _n}}}\frac{{4K}}{{\left( {n - 1} \right)\sin \left( {2\sqrt H r} \right)}}} \right)\left( {{\Delta _H}r} \right).\]
This estimate will be used to prove the Meyer's theorem.
\end{remark}
\begin{cor} If $Trace(A)$ is constant and $Trace(A)H \le Ric\left( {{\partial _r},A{\partial _r}} \right)$, then
\begin{eqnarray*}
{\Delta _{A,{f^A}}}r &\le& \frac{{Trace(A)}}{{n - 1}}\left( {1 + \frac{{4K}}{{Trace(A)}}} \right){\Delta _H}r,\\
{L_A}r &\le& \frac{{Trace(A)}}{{n - 1}}\left( {1 + \frac{{4K}}{{Trace(A)}}} \right)\left( {{\Delta _H}r} \right) + {\partial _r}.{f^A}(r).
\end{eqnarray*}
\end{cor}
Now we prove the  Meyer's theorem by using the so called excess functions. In fact this idea is used in \cite{petersen1998integral,wei2007comparison,wu2018myers}.  By adapting this approach we obtain the compactness result by applying the extended mean curvature Theorem \ref{Meancurvaturecomparison} for the elliptic differential operator ${{\Delta _{A,{f^A}}}}$ to the excess function.
\begin{proof}[\textbf{Proof of Theorem \ref{Meyer's theorem} (Meyer's theorem)}]We prove each item.
For (a) let $ p,q $ are two points in $ M $ with $dist\left( {p,q} \right) \ge \frac{\pi }{{\sqrt H }}$. Define $B: = dist\left( {p,q} \right) - \frac{\pi }{{\sqrt H }}$, ${r_1}(x): = dist\left( {p,x} \right)$ and ${r_2}(x): = dist\left( {q,x} \right)$. Also ${e_{p,q}}(x)$ be the excess function associated to the points $ p,q $. By triangle inequality, we have
${e_{p,q}}(x) \ge 0$ and ${e_{p,q}}\left( {\gamma (t)} \right) = 0$, where $ \gamma $ is the minimal geodesic joining $ p,q $. Hence ${\Delta _{A,{f^A}}}e\left( {\gamma (t)} \right) \ge 0$ in the barrier sense. Let ${y_1} = \gamma \left( {\frac{\pi }{{2\sqrt H }}} \right)$ and ${y_2} = \gamma \left( {B + \frac{\pi }{{2\sqrt H }}} \right)$. So ${r_i}\left( {{y_i}} \right) = \frac{\pi }{{2\sqrt H }} , i=1,2$. Remark \ref{largercomparison} concludes

\begin{equation}
{\Delta _{A,{f^A}}}({r_i})({y_i}) \le 2K\sqrt H.
\label{meyer1}
\end{equation}
Also from (\ref{bocnerin}) and the assumption on $Ric\left( {{\partial _r},A{\partial _r}} \right)$ we conclude,
\[{\Delta _{A,{f^A}}}r \le {\Delta _{A,{f^A}}}{r_0} - (n - 1){\delta _n}H\left( {r - {r_0}} \right).\]
Thus,
\begin{equation}
{\Delta _{A,{f^A}}}{r_1}({y_2}) \le {\Delta _{A,{f^A}}}{r_1}({y_1}) - B(n - 1){\delta _n}H.
\label{meyer2}
\end{equation}
So by (\ref{meyer1}) and (\ref{meyer2}) we have
\[0 \le {\Delta _{A,{f^A}}}\left( {{e_{p,q}}} \right)({y_2}) = {\Delta _{A,{f^A}}}{r_1}({y_2}) + {\Delta _{A,{f^A}}}{r_2}({y_2}) \le 4K\sqrt H  - B(n - 1){\delta _n}H,\]
so $B \le \frac{{4K}}{{{\delta _n}(n - 1)\sqrt H }}$ and
\[dist(p,q) \le \frac{\pi }{{\sqrt H }} + \frac{{4K}}{{{\delta _n}(n - 1)\sqrt H }}.\]
For the second part, Let $\bar M$ be the universal cover of $  M $ and $\Phi :\bar M \to M$ be the covering map. Then we define $\bar A: = {\Phi ^*}A = {\left( {{\Phi _*}} \right)^{ - 1}} \circ A \circ {\Phi _*}$. Note that for any vector field $X \in \Gamma \left( {TM} \right)$ with $\left| X \right| = 1$, we have
\[\sum\nolimits_i {\left\langle {{T^{\left( {{\nabla _{{\Phi ^*}X}}\bar A} \right)}}\left( {{\Phi ^*}{e_i},{\Phi ^*}X} \right),{\Phi ^*}{e_i}} \right\rangle }  = \sum\nolimits_i {\left\langle {{T^{\left( {{\nabla _X}A} \right)}}\left( {{e_i},X} \right),{e_i}} \right\rangle }  \circ \Phi  \le Hess{f^A}\left( {X,X} \right).\]
So by defining ${f^{\bar A}}: = {f^A} \circ \Phi $, we have
\[\sum\nolimits_i {\left\langle {{T^{\left( {{\nabla _{{\Phi ^*}X}}\bar A} \right)}}\left( {{\Phi ^*}{e_i},{\Phi ^*}X} \right),{\Phi ^*}{e_i}} \right\rangle }  \le Hess{f^{\bar A}}\left( {{\Phi ^*}X,{\Phi ^*}X} \right).\]
Also
\[\left| {{f^{\bar A}}} \right| = \left| {{f^A}} \right|\,\,\,\,\,\,\,\,\,and\,\,\,\,\,\,\bar Ric\left( {{\Phi ^*}X,\bar A\left( {{\Phi ^*}X} \right)} \right) = Ric\left( {X,AX} \right).\]
Thus the universal cover $ {\bar M}$  is compact and consequently $ M $ has finite fundamental group.
\end{proof}
To extended the inequalities of extended mean curvature comparison theorem on all of $ M $, we need the three definitions of inequalities in weak senses. The first one is the weak inequality in barrier sense which is originally defined  by Calabi \cite{calabi1958extension} in 1958.
\begin{definition}\label{barrier} Let $f \in {C^0}(M)$, $ X $ be a smooth vector field and $ A $ is bounded below as in Definition \ref{minmax}. We say that ${\Delta _{A,X}}u \ge v$ in the barrier sense, if for any point $ x_0 $ in $ M $ and any $\varepsilon> 0$ there exists a function ${u_{{x_0},\varepsilon }}$  which is called a support function and a neighborhood ${U_{{x_0},\varepsilon }}$ of $x_0$, such that the following properties are satisfied,
\begin{itemize}
\item[a)]${u_{{x_0},\varepsilon }} \in {C^2}({U_{{x_0},\varepsilon }})$,
\item[b)]${u_{{x_0},\varepsilon }}({x_0}) = u({u_{{x_0}}})$ and $u(x) \ge {u_{{x_0},\varepsilon }}(x)$ for all $x \in {U_{{x_0},\varepsilon }}$,
\item[c)]${\Delta _{A,X}}{u_{{x_0},\varepsilon }}({x_0}) \ge v - \varepsilon .$
\end{itemize}
Similarly, we say ${\Delta _{A,X}}u \le v$ in barrier sense, if ${\Delta _{A,X}}\left( { - u} \right) \ge  - v$ in the sense just defined.
\end{definition}

By \cite{calabi1958extension}  we know, If $\gamma $ is a minimal geodesic from $ p $ to $ q $, then for all $\varepsilon  > 0$ the function ${r_{q,\varepsilon }}(x) = \varepsilon  + dist\left( {\gamma (\varepsilon ),x} \right)$ is an upper barrier for the distance function $r(x) = dist\left( {p,x} \right)$. So we have the following inequality in barrier sense for the distance function. The following lemma is used for the extension of Quantitative Maximal Principle of Abresch and Gromoll in Proposition \ref{Quantitative Maximal Principle} and to get the same inequality in distribution sense in Lemma \ref{weakinequalitydis}.
\begin{lem}\label{weakinequalityba} Let $ p $ be the fixed point in $ M $ and $r(x) = dist\left( {p,x} \right)$.
If ${\Delta _{A,X}}(r) \le \alpha (r)$ point-wise on $M\backslash cut(p)$ for a continuous function $\alpha $ and  $v \in {C^2}(\mathbb{R})$ be non-negative, $u(x) = v(r(x))$ and suppose $v' \ge 0$, then
\begin{itemize}
\item[a)] If  $v' \ge 0$, then ${\Delta _{A,X}}(u) \le \left| {\nabla r} \right|_A^2v''(r) + \alpha (r)v'(r)$ in barrier sense on $ M $.
\item[b)]If  $v' \le 0$, then ${\Delta _{A,X}}(u) \ge \left| {\nabla r} \right|_A^2v''(r) + \alpha (r)v'(r)$ in barrier sense on $ M $.
\end{itemize}
Also a same results is valid for ${L_{A,f}}(r)$.
\end{lem}
 The second, is a weaker version and defined in the sense of viscosity which was introduced by Crandall and Lions in \cite{crandall1983viscosity}.

\begin{definition}\label{viscosity} Let $h\in {C^0}(M)$, we say that ${L_{A,f}}h(p) \ge a$ in
the viscosity sense, if for any $\phi  \in {C^2}\left( U \right)$ and neighborhood $ U $ of $ q $ with$\left( {h - \phi } \right)(q) = \mathop {\inf }\limits_U \left( {h - \phi } \right)$, we have ${L_{A,f}}h \le a$.
 \end{definition}
 By Lemma \ref{maximumpoint}, it is clear that barrier sub solutions are viscosity sub solutions.
The last very useful notion of inequality is the one in the sense of distribution.
\begin{definition}\label{distribution} For continuous functions $ u,v $  on manifold $ M $, we say ${L_{A,f}}(u) \le v$ in weak or distribution sense, if $\int_M {u{L_{A,f}}(\phi )} dvo{l_g} \le \int_M \phi  vvo{l_g}$ for each $\phi  \in Li{p_c}(M)$.
\end{definition}
When $ A $ is bounded from below as in Definition \ref{minmax}, it is known that, if $ u $ is a viscosity solution of  ${L_{A,f}}u \le h$ on $ M $ then, it is also a distribution solution and vice verse ( see \cite[theorem 3.2.11 ]{hormander1994notions} or \cite{lions1983optimal} ).
The following lemma is used for proving the monotonicity result for volume of geodesic balls in Theorems \ref{extended volume1}, \ref{extended volume comparison}.
\begin{lem}\label{weakinequalitydis}Let ${L_{A,f}}(r) \le \alpha (r)$ point-wise on $M\backslash cut(p)$ for a continuous function $\alpha $ and $v \in {C^2}(\mathbb{R})$ be non-negative. Let $u(x) = v(r(x))$, suppose $v' \ge 0$,  $v'' = 0$, then ${L_{A,f}}u \le v'\alpha (r)$  in distribution sense on $ M $.
\end{lem}
\begin{proof}

 By Lemma \ref{weakinequalityba}, we know the inequality is valid in barrier sense for ${L_{A,f}}(r)$. So it is valid in viscosity sense and by \cite{lions1983optimal} or \cite{hormander1994notions} is valid in distribution sense.
\end{proof}
\section{extended volume growth}\label{sec5}
In this section, we get some results about the growth of extended volume. We define the extended volume as follows.
\begin{definition}\label{Dev} Let $ M $ be a Riemannian manifold, $T \subseteq M$ which $\partial T$ has measure zero, $ A $ be a self adjoint (1,1)-tensor field on it, $ x_0 $ be e fixed point and $ r(x):=dist(x_0,x) $. We define the extended volume of geodesic ball $ B(x_0,R) $ as $vol^A(B(x_0,R): = \int_{B(x_0,R)} {\left\langle {A\nabla r,\nabla r} \right\rangle } dvo{l_g}$.
\end{definition}
 We compare this volume by the usual volume of geodesic balls in model spaces ${\mathbb{R}^n}\,,\,{\mathbb{S}^n}$ and ${\mathbb{H}^n}$. First, we give the following result.
\begin{theorem}\label{extended volume1}
Let $ M $ be a Riemannian manifold, $ x_0 $ be a fixed point and $ r(x):=dist(x_0,x) $, $ A $ be a self adjoint (1,1)-tensor field on it.  Assume ${L_A}r \le \frac{1}{C}\frac{{s{{n'}_H}(r)}}{{s{n_H}(r)}} + {\partial _r}.{f^A}(r)$ point-wise on $M\backslash cut(x_0)$ and the following condition holds,
\begin{equation}
\frac{1}{{C\left( {m  - 1} \right)}} \le \left\langle {\nabla r,A\nabla r} \right\rangle.
\label{volumecondition}
\end{equation}
Then
\begin{itemize}
\item[a)]For any $R$, we have,
\[\frac{d}{{dR}}\left( {\frac{{vol^A(B(x_0,R))}}{{vol_H^mB(R)}}} \right) \le \frac{1}{{{\delta _n}}}\frac{{{c_m}{R^{}}sn_H^{m - 1}(R)}}{{{{\left( {vol_H^mB(R)} \right)}^{1 + 1/p}}}}{\left( {\frac{{vol^A\left( {B(x_0,R)} \right)}}{{vol_H^mB(R)}}} \right)^{1 - 1/p}}{\left\| {\left( {\left( {\delta _n^{1/p}\left| {\nabla {f^A}} \right|} \right)} \right)} \right\|_{p,R}}.\]
\item[b)]For any $0 < R $, we get
\[{\left( {\frac{{vol^A(B(x_0,R))}}{{vol_H^mB(R)}}} \right)^{1/p}} - {\left( {\frac{{vol^A(B(x_0,R))}}{{vol_H^mB(r)}}} \right)^{1/p}} \le \frac{{{c_m}}}{{p{\delta _n}}}{\left\| {\left( {\delta _n^{1/p}\left( {\left| {\nabla {f^A}} \right|} \right)} \right)} \right\|_{p,R}}\int_r^R {\frac{{tsn_H^{m - 1}(t)}}{{{{\left( {vol_H^mB(t)} \right)}^{1 + 1/p}}}}dt.} \]
\item[c)]For any $0 < {r_1} \le {r_2} \le {R_1} \le {R_2} $, we have the following extended volume comparison result for annular regions,
\[\begin{array}{l}
 {\left( {\frac{{vol^A(B({x_0},{r_2},{R_2}))}}{{vol_H^mB({r_2},{R_2})}}} \right)^{1/p}} - {\left( {\frac{{vol^A(B({x_0},{r_1},{R_1}))}}{{vol_H^mB({r_1},{R_1})}}} \right)^{1/p}} \\
 \,\,\,\,\,\,\,\,\,\,\,\,\,\,\,\,\,\,\,\,\, \le \frac{{{c_m}}}{{p{\delta _n}}}{\left\| {\left( {\delta _n^{1/p}\left( {\left| {\nabla {f^A}} \right|} \right)} \right)} \right\|_{p,R}} \times \left[ {\int_{{R_1}}^{{R_2}} {\frac{{tsn_H^{m - 1}(t)}}{{{{\left( {vol_H^mB({r_2},t)} \right)}^{1 + 1/p}}}}} dt + \int_{{r_1}}^{{r_2}} {\frac{{{R_1}sn_H^{m - 1}({R_1})}}{{{{\left( {vol_H^mB(t,{R_1})} \right)}^{1 + 1/p}}}}dt} } \right] ,\\
 \end{array}\]
\end{itemize}
where ${vol_{H}^m (R)}$ is the volume of $ B(o,R)$ in the ${m }$-dimensional simply connected complete manifold with constant sectional curvature $ H $ and
\[{\left\| {\left( {\delta _n^{1/p}\left( {\left| {\nabla {f^A}} \right|} \right)} \right)} \right\|_{p,R}}: = \int_{ B({x_0},R)} {\delta _n^{}{{\left( {\left| {\nabla {f^A}} \right|} \right)}^p}} dvo{l_g}.\]
\end{theorem}
\begin{remark}\label{importantvolumerem} If $r \to 0$, then the integral
 \[\int_r^R {\frac{{tsn_H^{m - 1}(t)}}{{{{\left( {vol_H^mB(t)} \right)}^{1 + 1/p}}}}dt}, \]
 blows up.
\end{remark}
\begin{proof}[ Proof of the Theorem \ref{extended volume1}]
By assumption ${L_A}r \le \frac{1}{C}\frac{{s{{n'}_H}(r)}}{{s{n_H}(r)}} + {\partial _r}.{f^A}(r)$ point wise on $M\backslash cut(x_0)$ and by Lemma \ref{weakinequalitydis}, the inequality holds weakly (in distribution sense) on $M$. Thus, for every $0 \le \varphi  \in Li{p_c}(M)$, we have
\begin{equation}
 - \int_M {\left\langle {\nabla \varphi ,A\nabla r} \right\rangle dvo{l_g}}  \le \frac{1}{C}\int_M {\frac{{s{{n'}_H}\left( {r(x)} \right)}}{{s{n_H}\left( {r(x)} \right)}}\varphi dvo{l_g}}  + \int_M {{\partial _r}.{f^A}(r)\varphi dvo{l_g}} .
\label{mc1}
\end{equation}
For any $\varepsilon  > 0$, let ${\varphi _\varepsilon }(x)$ be the following radial cut-off function,
\[{\varphi _\varepsilon }(x) = {\rho _\varepsilon }(r(x))sn_H^{1 - m}(r(x)),\]
where ${\rho _\varepsilon }(r)$ is the function
\[{\rho _\varepsilon }(t) = \left\{ {\begin{array}{*{20}{c}}
   0 \hfill & {} \hfill & {t \in \left[ {0,r} \right),} \hfill  \\
   {\frac{{t - r}}{\varepsilon }} \hfill & {} \hfill & {t \in \left[ {r,r + \varepsilon } \right),} \hfill  \\
   1 \hfill & {} \hfill & {t \in \left[ {r + \varepsilon ,R - \varepsilon } \right),} \hfill  \\
   {\frac{{R - t}}{\varepsilon }} \hfill & {} \hfill & {t \in \left[ {R - \varepsilon ,R} \right),} \hfill  \\
   0 \hfill & {} \hfill & {t \in \left[ {R, + \infty } \right).} \hfill  \\
\end{array}} \right.\]
By computation we have,
\[\nabla {\varphi _\varepsilon } = {\rm{ }}\left\{ { - \frac{{{\chi _{R - \varepsilon ,R}}}}{\varepsilon } + \frac{{{\chi _{r,r + \varepsilon }}}}{\varepsilon } - \left( {m - 1} \right)\frac{{s{{n'}_H}(r(x))}}{{s{n_H}(r(x))}}{\rho _\varepsilon }} \right\}sn_H^{ - m + 1}(r(x))\nabla r,\]
for a.e.  $x \in M$, where ${\chi _{s,t}}$ is the characteristic function of the set $B(x_0,t)\backslash B(x_0,s)$. By inserting ${\varphi _\varepsilon }$ in to (\ref{mc1}) we get,
\begin{eqnarray*}
&&\frac{1}{\varepsilon }\int_{ \left( {B(x_0,R)\backslash B(x_0,R - \varepsilon )} \right)}{sn_H^{ - m + 1}(r(x))\left\langle {\nabla r,A\nabla r} \right\rangle dvo{l_g}}\\&&\,\,\,\, - \frac{1}{\varepsilon }\int_{ \left( {B(x_0,r + \varepsilon )\backslash B(x_0,r)} \right)}{sn_H^{ - m + 1}(r(x))\left\langle {\nabla r,A\nabla r} \right\rangle dvo{l_g}}
\\&&\,\,\,\, \le \int_{M } {\left( {\frac{1}{C} - \left( {m - 1} \right)\left\langle {\nabla r,A\nabla r} \right\rangle } \right)s{n_H}^\prime (r(x))sn_H^{ - m}(r(x)){\rho _\varepsilon }dvo{l_g}} \\&&\,\,\,\, + \int_{M } {{\partial _r}.{f^A}(r){\rho _\varepsilon }sn_H^{ - m + 1}(r(x))dvo{l_g}.}
\end{eqnarray*}

So if, (note that $s{n_H}^\prime (r(x))sn_H^{ - m}(r(x)) \ge 0$)
\[\left( {\frac{1}{C}  - \left( {m  - 1} \right)\left\langle {\nabla r,A\nabla r} \right\rangle } \right) \le 0 \Leftrightarrow \frac{1 }{{C\left( {m  - 1} \right)}} \le \left\langle {\nabla r,A\nabla r} \right\rangle, \]
then
\begin{eqnarray*}
&&\frac{1}{\varepsilon }\int_{ \left( {B(x_0,R)\backslash B(x_0,R - \varepsilon )} \right)} sn_H^{ - m + 1}(r(x))\left\langle {\nabla r,A\nabla r} \right\rangle dvo{l_g}
\\&&\,\,\,\,- \frac{1}{\varepsilon }\int_{ \left( {B(x_0,r + \varepsilon )\backslash B(x_0,r)} \right)} {sn_H^{ - m + 1}(r(x))\left\langle {\nabla r,A\nabla r} \right\rangle dvo{l_g}}
\\&&\,\,\,\, \le \int_{M } {\partial _r}.{f^A}(r){\rho _\varepsilon }sn_H^{ - m + 1}(r(x))dvo{l_g}.
\end{eqnarray*}

Letting $\varepsilon  \to 0$, we conclude,
\begin{eqnarray*}
\frac{{vol^A(\partial B(x_0,R))}}{{sn_H^{m - 1}(R)}} - \frac{{vol^A(\partial B(x_0,r))}}{sn_H^{m - 1}(r)}&\le& {\smallint _{ \left( {B(x_0,R)\backslash B(x_0,r)} \right)}}{\partial _r}.{f^A}(r)sn_H^{ - m + 1}(r(x))dvo{l_g}\\&\le& sn_H^{ - m + 1}(r){\smallint _{ \left( {B(x_0,R)\backslash B(x_0,r)} \right)}}\left| {\nabla {f^A}(r)} \right|dvo{l_g}.
\end{eqnarray*}
Let $p > 1$, using H\"{o}lder inequality, we obtain
\begin{eqnarray}\label{Holderinequality}
  && {sn_H^{m - 1}(r)vol_{}^A(\partial B(x_0,R)) - sn_H^{m - 1}(R)vol^A(\partial B(x_0,R))}  \\&&\nonumber\,\,\,\,
   { \le \frac{1}{{{\delta _n}}}sn_H^{m - 1}(R){{\left( {vol^A\left( {B(x_0,R)} \right)} \right)}^{1 - \left( {1/p} \right)}}{{\left( {\int_{ B(x_0,R)} {{\delta _n}{{\left( {\left| {\nabla {f^A}} \right|} \right)}^p}dvo{l_g}} } \right)}^{1/p}}.}
\end{eqnarray}
 So, we have
\begin{eqnarray*}
&&\frac{d}{{dR}}\left( {\frac{{vo{l^A}(B({x_0},R))}}{{vol_H^mB(R)}}} \right)\\
&=&\frac{{vol_H^mB(R)vo{l^A}(\partial B({x_0},R)) - vol_H^m\left( {\partial B(R)} \right)vo{l^A}(B({x_0},R))}}{{{{\left( {vol_H^mB(R)} \right)}^2}}}
\\&\le&{c_m}{{\left( {vol_H^mB(R)} \right)}^{ - 2}}\int_0^R {\frac{1}{{{\delta _n}}}sn_H^{m - 1}(R){{\left( {vo{l^A}\left( {B({x_0},R)} \right)} \right)}^{1 - \left( {1/p} \right)}}dr}
\\&&\,\,\,\,\,\,\,\,\,\,\,\,\,\,\,\,\,\,\,\,\,\,\,\,\,\,\,\,\,\,\,\,\,\,\,\,\,\,\,\,\,\,\,\,\,\,\,\,\,\,\,\,\,\,\,\,\,\,\,\,\, \times {{\left( {\int_{B({x_0},R)} {{\delta _n}{{\left( {\left| {\nabla {f^A}} \right|} \right)}^p}dvo{l_g}} } \right)}^{1/p}}
\\&\le&\frac{1}{{{\delta _n}}}\frac{{{c_m}Rsn_H^{m - 1}(R)}}{{{{\left( {vol_H^mB(R)} \right)}^{1 + 1/p}}}}{{\left( {\frac{{vo{l^A}\left( {B({x_0},R)} \right)}}{{vol_H^mB(R)}}} \right)}^{1 - 1/p}}{{\left( {\int_{B({x_0},R)} {{\delta _n}{{\left( {\left| {\nabla {f^A}} \right|} \right)}^p}dvo{l_g}} } \right)}^{1/p}}
\\&\le&\frac{1}{{{\delta _n}}}\frac{{{c_m}Rsn_H^{m - 1}(R)}}{{{{\left( {vol_H^mB(R)} \right)}^{1 + 1/p}}}}{{\left( {\frac{{vo{l^A}\left( {B({x_0},R)} \right)}}{{vol_H^mB(R)}}} \right)}^{1 - 1/p}}{{\left\| {\left( {\delta _n^{1/p}\left( {\left| {\nabla {f^A}} \right|} \right)} \right)} \right\|}_{p,R}}.
\end{eqnarray*}
Thus,
\begin{eqnarray*}
\frac{d}{{dR}}\left( {{{\left( {\frac{{vol^A(B(x_0,R))}}{{vol_H^mB(R)}}} \right)}^{1/p}}} \right)
& =& \frac{1}{p}{{\left( {\frac{{vol^A(B(x_0,R))}}{{vol_H^mB(R)}}} \right)}^{ - 1 + 1/p}}\frac{d}{{dR}}\left( {\frac{{vol^A(B(x_0,R))}}{{vol_H^mB(R)}}} \right)\\
&\le&\frac{{{c_m}}}{{p{\delta _n}}}{{\left\| {\left( {\delta _n^{1/p}\left( {\left| {\nabla {f^A}} \right|} \right)} \right)} \right\|}_{p,R}}\frac{R}{{{{\left( {vol_H^mB(R)} \right)}^{1 + 1/p}}}}sn_H^{m - 1}(R).
\end{eqnarray*}
Consequently,
\[{\left( {\frac{{vol^A(B(x_0,R))}}{{vol_H^mB(R)}}} \right)^{1/p}} - {\left( {\frac{{vol^A(B(x_0,R))}}{{vol_H^mB(r)}}} \right)^{1/p}} \le \frac{{{c_m}}}{{p{\delta _n}}}{\left\| {\left( {\delta _n^{1/p}\left( {\left| {\nabla {f^A}} \right|} \right)} \right)} \right\|_{p,R}}\int_r^R {\frac{{tsn_H^{m - 1}(t)}}{{{{\left( {vol_H^mB(t)} \right)}^{1 + 1/p}}}}dt.} \]
For the volume comparison for annular regoins we use twice procedures of part (a).
At first note
\begin{eqnarray*}
&&\frac{d}{{dR}}\left( {\frac{{vol^A(B({x_0},r,R))}}{{vol_H^mB(r,R)}}} \right)\\
&=& {c_m}\frac{{vol^A(\partial B({x_0},R))\int_r^R {sn_H^{m - 1}(t)dt}  - sn_H^{m - 1}(R)vol^A(B({x_0},r,R))}}{{{{\left( {vol_H^mB(r,R)} \right)}^2}}}\\
&=&{{\left( {vol_H^mB(r,R)} \right)}^{ - 2}}\int_r^R\big[ {{c_m}sn_H^{m - 1}(t)vol^A(\partial B({x_0},R)) - {c_m}sn_H^{m - 1}(R)vol^A(B({x_0},t))\big]dt}
\\&\le&  \frac{1}{{{\delta _n}}}\frac{{{c_m}Rsn_H^{m - 1}(R)}}{{{{\left( {vol_H^mB(r,R)} \right)}^{1 + 1/p}}}}{{\left( {\frac{{vol^A\left( {B({x_0},r,R)} \right)}}{{vol_H^mB(r,R)}}} \right)}^{1 - 1/p}}{{\left\| {\left( {\delta _n^{1/p}\left( {\left| {\nabla {f^A}} \right|} \right)} \right)} \right\|}_{p,R}},
\end{eqnarray*}
where we use parameter ${vol^A\left( {B({x_0},r,R)} \right)}$ instead of $vo{l^A}\left( {B({x_0},R)} \right)$ in the inequality (\ref{Holderinequality}).
and similarly,
\[\frac{d}{{dR}}{\left( {\frac{{vol^A(B({x_0},r,R))}}{{vol_H^mB(r,R)}}} \right)^{1/p}} \le \frac{1}{{p{\delta _n}}}\frac{{{c_m}Rsn_H^{m - 1}(R)}}{{{{\left( {vol_H^mB(r,R)} \right)}^{1 + 1/p}}}}{\left\| {\left( {\delta _n^{1/p}\left( {\left| {\nabla {f^A}} \right|} \right)} \right)} \right\|_{p,R}}.\]
Thus,
\begin{eqnarray}\label{anulusvolume1}
 &&{\left( {\frac{{vol^A(B({x_0},{r_2},{R_2}))}}{{vol_H^mB({r_2},{R_2})}}} \right)^{1/p}} - {\left( {\frac{{vol^A(B({x_0},{r_2},{R_1}))}}{{vol_H^mB({r_2},{R_1})}}} \right)^{1/p}} \\ \nonumber&&
\,\,\,\,\,\,\,\,\,\, \le \frac{{{c_m}}}{{p{\delta _n}}}{\left\| {\left( {\delta _n^{1/p}\left( {\left| {\nabla {f^A}} \right|} \right)} \right)} \right\|_{p,R}}\int_{{R_1}}^{{R_2}} {\frac{{tsn_H^{m - 1}(t)}}{{{{\left( {vol_H^mB({r_2},t)} \right)}^{1 + 1/p}}}}} dt.
\end{eqnarray}

Similarly,
\begin{eqnarray*}
&&\frac{d}{{dr}}\left( {\frac{{vol^A(B({x_0},r,R))}}{{vol_H^mB(r,R)}}} \right)
\\&=& {c_m}\frac{{ - vol^A(\partial B({x_0},r))\int_r^R {sn_H^{m - 1}(t)dt}  - \left( { - sn_H^{m - 1}(r)} \right)vol^A(B({x_0},r,R))}}{{{{\left( {vol_H^mB(r,R)} \right)}^2}}}
\\&=&{{\left( {vol_H^mB(r,R)} \right)}^{ - 2}}\int_r^R {{c_m}sn_H^{m - 1}(r)vol^A(\partial B({x_0},t)) - {c_m}sn_H^{m - 1}(t)vol^A(\partial B({x_0},r))dt}
\\&\le&\frac{1}{{{\delta _n}}}\frac{{{c_m}sn_H^{m - 1}(R){{\left\| {\left( {\delta _n^{1/p}\left( {\left| {\nabla {f^A}} \right|} \right)} \right)} \right\|}_{p,R}}}}{{{{\left( {vol_H^mB(r,R)} \right)}^{1 + 1/p}}}}{{\left( {\frac{{vol^A\left( {B({x_0},r,R)} \right)}}{{vol_H^mB(r,R)}}} \right)}^{1 - 1/p}}.
\end{eqnarray*}
So,
\[\frac{d}{{dr}}\left( {{{\left( {\frac{{vol^A(B({x_0},r,R))}}{{vol_H^mB(r,R)}}} \right)}^{1/p}}} \right) \le \frac{{{c_m}Rsn_H^{m - 1}(R)}}{{p{\delta _n}}}\frac{{{{\left\| {\left( {\delta _n^{1/p}\left( {\left| {\nabla {f^A}} \right|} \right)} \right)} \right\|}_{p,R}}}}{{{{\left( {vol_H^mB(r,R)} \right)}^{1 + 1/p}}}}.\]
And
\begin{eqnarray}\label{anuluavolume2}
&& {\left( {\frac{{vol^A(B({x_0},{r_2},{R_1}))}}{{vol_H^mB({r_2},{R_1})}}} \right)^{1/p}} - {\left( {\frac{{vol^A(B({x_0},{r_1},{R_1}))}}{{vol_H^mB({r_1},{R_1})}}} \right)^{1/p}} \\ \nonumber
&&\,\,\,\,\,\, \le \frac{{{c_m}}}{{p{\delta _n}}}{\left\| {\left( {\delta _n^{1/p}\left( {\left| {\nabla {f^A}} \right|} \right)} \right)} \right\|_{p,R}}\int_{{r_1}}^{{r_2}} {\frac{{{R_1}sn_H^{m - 1}({R_1})}}{{{{\left( {vol_H^mB(t,{R_1})} \right)}^{1 + 1/p}}}}dt} .
\end{eqnarray}
By adding (\ref{anuluavolume2}) and (\ref{anulusvolume1}), we get
\begin{eqnarray*}
&& {\left( {\frac{{vol^A(B({x_0},{r_2},{R_2}))}}{{vol_H^mB({r_2},{R_2})}}} \right)^{1/p}} - {\left( {\frac{{vol^A(B({x_0},{r_1},{R_1}))}}{{vol_H^mB({r_1},{R_1})}}} \right)^{1/p}} \\
&&\,\,\,\, \le \frac{{{c_m}}}{{p{\delta _n}}}{\left\| {\left( {\delta _n^{1/p}\left( {\left| {\nabla {f^A}} \right|} \right)} \right)} \right\|_{p,R}}  \left[ {\int_{{R_1}}^{{R_2}} {\frac{{tsn_H^{m - 1}(t)}}{{{{\left( {vol_H^mB({r_2},t)} \right)}^{1 + 1/p}}}}} dt + \int_{{r_1}}^{{r_2}} {\frac{{{R_1}sn_H^{m - 1}({R_1})}}{{{{\left( {vol_H^mB(t,{R_1})} \right)}^{1 + 1/p}}}}dt} } \right].
\end{eqnarray*}
\end{proof}
As the classical case, we can get an extension of Yau's theorem. It was originally proved by analytic methods in the Riemannian case for Ricci tensor by Calabi and Yau  in 1976 \cite{yau1976some}. We adapted \cite{wu2016comparison} for the proof.
\begin{theorem}[\textbf{Extension of Yau theorem}]\label{yau}
 Let $M$ be  non compact, $ x_0 $ be a fixed point and $Ri{c_{Trace(A)}}\left( {{\partial _r},A{\partial _r}} \right) \ge 0$, then for any $p > n$ and $R \ge 2$, there is an $\varepsilon  = \varepsilon (m,p,{\delta _1},{\delta _n},R)$ such that if
\[\mathop {\sup }\limits_{x \in M} \frac{{\delta _n^{}}}{{{\delta _1}vol(B(x,r))}}{\left( {\int_{B(x,r)} {{{\left| {\nabla {f^A}} \right|}^p}dvo{l_g}} } \right)^{1/p}} < \varepsilon \]
then
  \[vol\left( {B({x_0},r)} \right) \ge cr,\]
where $c$ is a constant.
\end{theorem}
\begin{proof}Let $ x \in M $ be a point with $ {\rm{dist}}\left( {{x_0},x} \right){\rm{ }} = {\rm{ R }} \ge {\rm{ 2}} $. By relative comparison Theorem \ref{extended volume1} for annulus and letting  ${r_1} = 0\,,\,{r_2} = R - 1\,,\,{R_1} = R$ and ${R_2} = R + 1$ we have,
\[{\left( {\frac{{vo{l^A}\left( {B(x,R - 1,R + 1)} \right)}}{{{{\left( {t + R} \right)}^m} - {{\left( {R - 1} \right)}^m}}}} \right)^{1/p}} - {\left( {\frac{{vo{l^A}\left( {B(x,0,R)} \right)}}{{{R^m}}}} \right)^{1/p}} \le \frac{{2{c_m}}}{{p{\delta _n}}}{\left( {R + 1} \right)^{m + 1}}{\left\| {\left( {\delta _n^{1/p}\left( {\left| {\nabla {f^A}} \right|} \right)} \right)} \right\|_{p,M,R + 1}}.\]
So
\[\left( {\frac{{vo{l^A}\left( {B(x,R - 1,R + 1)} \right)}}{{{{\left( {R + 1} \right)}^m} - {{\left( {R - 1} \right)}^m}}}} \right) \le \left( {\frac{{vo{l^A}\left( {B(x,R)} \right)}}{{{R^m}}}} \right) + C{\left( {R + 1} \right)^{\left( {m + 1} \right)p}}\left\| {\left( {\delta _n^{1/p}\left( {\left| {\nabla {f^A}} \right|} \right)} \right)} \right\|_{p,M,R + 1}^p .\]
By multiplying with $\frac{{{{\left( {R + 1} \right)}^m} - {{\left( {R - 1} \right)}^m}}}{{vo{l^A}\left( {B(x,R + 1)} \right)}}$, we get
\begin{eqnarray*}
&&\left( {\frac{{vo{l^A}\left( {B(x,R - 1,R + 1)} \right)}}{{{{\left( {R + 1} \right)}^m} - {{\left( {R - 1} \right)}^m}}}} \right)
\\ &\le&\frac{D}{R} + D{\left( {R + 1} \right)^{\left( {m + 1} \right)p}}\mathop {\sup }\limits_{x \in M} \frac{{\delta _n^{}}}{{{\delta _1}vo{l^m}(B(x,R + 1))}}{\left( {\int_{B(x,R + 1)} {{{\left| {\nabla {f^A}} \right|}^p}dvo{l_g}} } \right)^{1/p}},
\end{eqnarray*}

where $ D $ is some constant depends on $ m,p,\delta_1 $. We choose $\varepsilon  = \varepsilon (m,p,{\delta _1},{\delta _n},R)$  small enough such that
\[\frac{{vo{l^A}\left( {B(x,R - 1,R + 1)} \right)}}{{vo{l^A}\left( {B(x,R + 1)} \right)}} \le \frac{{2D}}{R}.\]
For $R \ge 2$ we have,
\[vo{l^A}\left( {B({x_0},2R + 1)} \right) \ge \frac{{vo{l^A}\left( {B({x_0},1)} \right)}}{{2D}}R.\]
\end{proof}
To get a uniform volume comparison result, we need to bound the integral
\[\int_r^R {\frac{{tsn_H^{m - 1}(t)}}{{{{\left( {vol_H^mB(t)} \right)}^{1 + 1/p}}}}dt} .\]
We can achieve this goal with putting some restrictions on $ A $, by the following definition.
\begin{definition} We call the tensor $ A $ satisfies measure condition on a ball $B\left( {{x_0},{r_0}} \right)$, if there exists some almost smooth function ${h^A}$ which satisfies the following,
\begin{equation}
\nabla {h^A}(x) = \left\{ {\begin{array}{*{20}{c}}
   {{A^{ - 1}} \circ \nabla {f^A}(x)} & {x \in B({x_0},{r_0})}  \\
   0 & {O.W.}  \\
\end{array}} \right.
\label{invariantmeasure}
\end{equation}
We call ${e^{ - {h^A}}}dvo{l_g}$ and $ r_0 $ as corresponding measure and radius measure to the tensor $ A $, respectively.
Also we define the function ${{\bar f}^A}$ as
\[{{\bar f}^A}: = \left\{ {\begin{array}{*{20}{c}}
   0 \hfill & {x \in B({x_0},{r_0})} \hfill  \\
   {{f^A}} \hfill & {O.W.} \hfill  \\
\end{array}} \right.\]
\end{definition}
\begin{example} If ${f^A} = 0$ on some ball ${B({x_0},{r_0})}$, ( for example, when ${\nabla ^2}A = 0$ on ball ${B({x_0},{r_0})}$) then $ r_0 $ is one radius measure and $dvo{l_g}$ is the measure corresponding to $ A $.
\end{example}
\begin{proof}[\textbf{Proof of Theorem  \ref{extended volume comparison} (Extended volume comparison)}]To prove parts (a) and (b), note that by Theorem \ref{Meancurvaturecomparison} pat (b),  when

 \[m = C\left( {{\delta _n},n,{\delta _1},K,K',H} \right) = \left[ {\frac{{{\delta _n}(n - 1) + 4(K + K')}}{{{\delta _1}}}} \right] + 2,\]

then the conditions of Theorem \ref{extended volume1} are satisfied. To prove the part (c), a simply computation shows that ${L_{A,{h^A}}}r = {L_A}r - \left\langle {\nabla {h^A},A\nabla r} \right\rangle$, so by the conditions of the theorem we have,
\begin{eqnarray*}
{L_{A,{h^A}}}r&\le&{\delta _n}\left( {1 + \frac{{4\left( {K + K'} \right)}}{{{\delta _n}\left( {n - 1} \right)}}} \right)\left( {{\Delta _H}r} \right) + \left\langle {\nabla {f^A} - A\nabla {h^A},{\partial _r}} \right\rangle
\\&=&{\delta _n}\left( {1 + \frac{{4\left( {K + K'} \right)}}{{{\delta _n}\left( {n - 1} \right)}}} \right)\left( {{\Delta _H}r} \right) + \left\langle {\nabla {{\bar f}^A},{\partial _r}} \right\rangle .
\end{eqnarray*}
Similar argument shows that,
\[\frac{{vol_{{h^A}}^A(\partial B({x_0},R))}}{{sn_H^{m - 1}(R)}} - \frac{{vol_{{h^A}}^A(\partial B({x_0},r))}}{{sn_H^{m - 1}(r)}} \le sn_H^{ - m + 1}(r)\int_{\left( {B({x_0},R)\backslash B({x_0},r)} \right)} {\left| {\nabla {{\bar f}^A}(r)} \right|{e^{ - {h^A}}}dvo{l_g}} .\]
Let $p > 1$, using H\"{o}lder inequality, we obtain
\begin{eqnarray*}
   &&{sn_H^{m - 1}(r)vol_{{h^A}}^A(\partial B(x_0,R)) - sn_H^{m - 1}(R)vol_{{h^A}}^A(\partial B(x_0,r))} \hfill  \\&&
   {\,\,\,\,\,\,\,\,\,\,\,\, \le sn_H^{m - 1}(R){{\left( {\int_{ B(x_0,R)} {{{ {{r}} }^{1/\left( {p - 1} \right)}}{e^{ - {h^A}}}dvo{l_g}} } \right)}^{1 - 1/p}}{{\left( {\int_{ B(x_0,R)} {{{\left( {{r^{ - 1/p}}\left| {\nabla {{\bar f}^A}} \right|} \right)}^p}{e^{ - {h^A}}}dvo{l_g}} } \right)}^{1/p}}} \hfill  \\&&
   {\,\,\,\,\,\,\,\,\,\,\,\, \le \frac{1}{{{\delta _n}}}{R^{1/p}}sn_H^{m - 1}(R){{\left( {vol_{{h^A}}^A\left( {B(x_0,R)} \right)} \right)}^{1 - \left( {1/p} \right)}}{{\left( {\int_{B(x_0,R_T)} {{\delta _n}{{\left( {{r^{ - 1/p}}\left| {\nabla {{\bar f}^A}} \right|} \right)}^p}{e^{ - {h^A}}}dvo{l_g}} } \right)}^{1/p}}.} \hfill
\end{eqnarray*}
\[\begin{array}{*{20}{c}}

\end{array}\]
 So, we have
\[\frac{d}{{dR}}\left( {\frac{{vol_{{h^A}}^A(B({x_0},R))}}{{vol_H^mB(R)}}} \right) \le \frac{1}{{{\delta _n}}}\frac{{{c_m}{R^{1 + 1/p}}sn_H^{m - 1}(R)}}{{{{\left( {vol_H^mB(R)} \right)}^{1 + 1/p}}}}{\left( {\frac{{vol_{{h^A}}^A\left( {B({x_0},R)} \right)}}{{vol_H^mB(R)}}} \right)^{1 - 1/p}}{\left\| {\left( {\delta _n^{1/p}\left( {{r^{ - 1/p}}\left| {\nabla {{\bar f}^A}} \right|} \right)} \right)} \right\|_{p,{R_T}}}.\]
Similarly
\[\frac{d}{{dR}}\left( {{{\left( {\frac{{vol_{{h^A}}^A(B({x_0},R))}}{{vol_H^mB(R)}}} \right)}^{1/p}}} \right) \le \frac{{{c_m}}}{{p{\delta _n}}}{\left\| {\left( {\delta _n^{1/p}\left( {{r^{ - 1/p}}\left| {\nabla {{\bar f}^A}} \right|} \right)} \right)} \right\|_{p,{R}}}{\left( {\frac{R}{{\left( {vol_H^mB(R)} \right)}}} \right)^{1 + 1/p}}sn_H^{m - 1}(R).\]
By integrating and for $p > m$ we obtain,
\begin{eqnarray*}
  && {{{\left( {\frac{{vol_{{h^A}}^A(B({x_0},R)}}{{vol_H^mB(R)}}} \right)}^{1/p}} - {{\left( {\frac{{vol_{{h^A}}^A(B({x_0},r)}}{{vol_H^mB(r)}}} \right)}^{1/p}}} \hfill  \\&&
   {\,\,\,\,\,\,\,\,\,\,\,\,\,\,\,\,\,\,\,\,\,\,\,\,\,\,\,\,\,\,\, \le \frac{{{c_m}}}{{p{\delta _n}}}{{\left\| {\left( {\delta _n^{1/p}\left( {{r^{ - 1/p}}\left| {\nabla {{\bar f}^A}} \right|} \right)} \right)} \right\|}_{p,{R_T}}}\int_r^R {{{\left( {\frac{r}{{\left( {vol_H^mB(r)} \right)}}} \right)}^{1 + 1/p}}sn_H^{m - 1}(r)dr} } \hfill  \\&&
   {\,\,\,\,\,\,\,\,\,\,\,\,\,\,\,\,\,\,\,\,\,\,\,\,\,\,\,\,\,\,\, \le \frac{{{c_m}}}{{p{\delta _n}}}{{\left\| {\left( {\delta _n^{1/p}\left( {{r^{ - 1/p}}\left| {\nabla {{\bar f}^A}} \right|} \right)} \right)} \right\|}_{p,{R_T}}}\int_0^R {{{\left( {\frac{r}{{\left( {vol_H^mB(r)} \right)}}} \right)}^{1 + 1/p}}sn_H^{m - 1}(r)dr} } \hfill  \\&&
   {\,\,\,\,\,\,\,\,\,\,\,\,\,\,\,\,\,\,\,\,\,\,\,\,\,\,\,\,\,\,\, \le \tilde C\left( {{\delta _n},n,{\delta _1},K,K',H,R,p} \right){{\left\| {\left( {\delta _n^{1/p}\left( {{r^{ - 1/p}}\left| {\nabla {{\bar f}^A}} \right|} \right)} \right)} \right\|}_{p,{R_T}}}.} \hfill  \\
\end{eqnarray*}
where
\[\tilde C\left( {{\delta _n},n,{\delta _1},K,K',H,p,R} \right) = \frac{{{c_m}}}{{p{\delta _n}}}\int_0^R {{{\left( {\frac{r}{{\left( {vol_H^mB(r)} \right)}}} \right)}^{1 + 1/p}}sn_H^{m - 1}(r)dr}.\]
With a similar approach to \cite{{hu2008bounds}}, by choosing
\[\bar C\left( {{\delta _n},n,{\delta _1},K,K',H,p,{R_T}} \right) = \tilde C\left( {{\delta _n},n,{\delta _1},K,K',H,p,{R_T}} \right){\left( {vol_H^m\left( {{R_T}} \right)} \right)^{1/p}},\]
and
\[\varepsilon (p,A,{R_T}) = \frac{1}{{{{\left( {vol_{{h^A}}^A\left( {B({x_0},{R_T})} \right)} \right)}^{1/p}}}}{\left\| {\left( {\delta _n^{1/p}\left( {{r^{ - 1/p}}\left| {\nabla {{\bar f}^A}} \right|} \right)} \right)} \right\|_{p,{R_T}}},\]
when
$\varepsilon  < \frac{1}{{\bar C\left( {{\delta _n},n,{\delta _1},K,K',H,p,{R_T}} \right)}}$, one has
\[{\left( {\frac{1}{{vol_{{h^A}}^A(B({x_0},R))}}} \right)^{1/p}} \le \frac{1}{{1 - \bar C\varepsilon }}{\left( {\frac{{vol_H^m({R_T})}}{{vol_H^m(R)vol_{{h^A}}^A(R)}}} \right)^{1/p}}.\]
Also, when $\varepsilon  < \frac{1}{{2\bar C}}$ then,
\[\frac{{vol_{{h^A}}^A(B({x_0},R)}}{{vol_{{h^A}}^A(B({x_0},r)}} \le {\left( {\frac{{1 - \bar C\varepsilon }}{{1 - 2\bar C\varepsilon }}} \right)^p}\frac{{vol_H^mB(R)}}{{vol_H^mB(r)}}.\]
\end{proof}
\begin{remark}Note that ${\left( {\frac{{1 - \bar C\varepsilon }}{{1 - 2\bar C\varepsilon }}} \right)^p}$ is non decreasing both in ${R_T}$ and $ \varepsilon $.
\end{remark}
\section{Cheeger-Gromoll splitting theorem}\label{sec7}
One of the important application of the mean curvature comparison is the Cheeger-Gromoll splitting theorem. In this section we prove an extension of Cheeger-Gromoll splitting theorem by restriction on the extended Ricci tensor $Ric\left( {X,AX} \right)$. The approach is similar to the original one. i.e, we show that for the Bussemann function $b_\gamma ^ + $ associated to the ray $ {\gamma _ + } $, the vector field $\nabla b_\gamma ^ +$ is Killing and $\left\| {\nabla b_\gamma ^ + } \right\| = 1$. To do this, we use the extended Bochner formula and the restriction on the extended Ricci tensor to show $b_\gamma ^ + $ is a harmonic function.
At first, one should provide the maximum principal for the operator ${\Delta _{A,X}} = {\Delta _A} - \left\langle {X,\nabla \,\,} \right\rangle $. To do this we get the following Lemma.
\begin{lem}\label{maximumpoint}Let $f,h \in {C^2}(M)$ and $p \in M$ and $U$ be a neighborhood of $p$. If
\begin{itemize}
\item[a)]$f(p) = h(p)$,
\item[b)]$f(x) \ge h(x)$  for all $x \in U$,
\end{itemize}
then
\begin{itemize}
\item[a)]$\nabla f(p) = \nabla h(p)$,
\item[b)]$Hessf(p) \ge Hessh(p)$,
\item[c)]${\Delta _{A,X}}f(p) \ge {\Delta _{A,X}}h(p)$.
\end{itemize}
\end{lem}
\begin{proof}
Parts (a) and (b) are clear. For part (c) its sufficient to show that  ${\Delta _A}f(p) \ge {\Delta _A}h(p)$. We know that $Hess\left( {f - h} \right)(p) \ge 0$. By assumption, $A$ is positive definite, so $A=B^2$ and
\begin {eqnarray*}
{\Delta _A}\left( {f - h} \right)(p)&=&Trace\left( {{B^2} \circ hess\left( {f - h} \right)(p)} \right) = Trace\left( {B \circ hess\left( {f - h} \right)(p) \circ B} \right)
\\&=&\sum\nolimits_i {\left\langle {hess\left( {f - h} \right)(p) \circ B{e_i},B{e_i}} \right\rangle }  \ge 0
   \end{eqnarray*}
so ${\Delta _{A,X}}f(p) \ge {\Delta _{A,X}}h(p)$.
\end{proof}
Now we extend the maximum principle for ${\Delta _A}$.
\begin{theorem}[\textbf{Extended maximum principle}] \label{Maximum principle} Let $f \in {C^0}(M)$ and ${\Delta _{A,X}}f \ge 0$ in barrier sense, then $f$ is constant in a neighborhood of each local maximum. So, if $f$ has a global maximum, then $f$ is constant.
\end{theorem}
\begin{proof}Let $p \in M$ be a local maximum.  If ${\Delta _{A,X}}f (p)> 0$, then we have a contradiction by Lemma \ref{maximumpoint} part (c). So we assume ${\Delta_{A,X}}f (p)\ge 0$. Without loose of generality we can assume that there is a sufficiently small $r < inj(p)$, such that $ p $ is the maximum of restricted function $f:B(x_0,r) \to\mathbb{R} $ and there is some point ${x_0} \in \partial B(p,r)$ such that $f({x_0}) \ne f(p)$. As usual, we define
\[V: = \left\{ {x \in \partial B(p,r):f(x) = f(p)} \right\}.\]
Let $ U $ be an open neighborhood with the property $V \subseteq U \subseteq \partial B(p,r)$ and $\phi$ be the function such that,
\[\phi (p) = 0\,\,,\,\,\,{\left. \phi  \right|_U} < 0\,\,,\,\,\nabla \phi  \ne 0.\]
Then for the function $h = {e^{\alpha \phi }} - 1$ we have,
\[{\Delta _{A,X}}h = \alpha {e^{\alpha \phi }}\left( {\alpha \left\langle {\nabla \phi ,A\nabla \phi } \right\rangle  + {\Delta _{A,X}}\phi } \right) \ge \alpha {e^{\alpha \phi }}\left( {\alpha {\delta _1}{{\left| {\nabla \phi } \right|}^2} + {\Delta _{A,X}}\phi } \right).\]
So by choosing $\alpha$ large enough we have,
\[{\Delta _{A,X}}h > 0\,\,,\,\,{\left. h \right|_V}{\rm{ < }}0\,\,,\,\,h(p) = 0.\]
Now by defining $F = f + \delta h$, for small enough $\delta  > 0$, we have ${\Delta_{A,X}}F > 0$ on $\overline {B(x_0,r)} $ and $ F $ has a maximum point in $B(x_0,r)$, which is a contradiction by first part of the proof.
 \end{proof}

The regularity is provided by the following.
\begin{prop}\label{regular} If $ A $ is bounded ( ${\delta _1} > 0 $ ) and ${\Delta _{A,X}}f = 0$ in the barrier sense, then f is smooth.
\end{prop}
\begin{proof} Regularity is a local property and $A$ is self-adjoint positive definite. ${\Delta _{A,X}}$ satisfied the elliptic conditions in section 6.3-6.4 or Theorem 6.17 of \cite{gilbarg2015elliptic}. So by elliptic regularity property $f$ is smooth ( Note that by our definition when $ A $ is bounded, so for each vector field $X$ with $\left| X \right| = 1$ one has $0 < {\delta _1} \le \left\langle {AX,X} \right\rangle$).
\end{proof}
Now we ready to prove the extension of Cheeger-Gromoll splitting theorem. To this end, we use the so-called Busseman functions of a line in a manifold and show that gradient of a Busseman funtion ${b^ + }$ is Killing and $\left\| {\nabla {b^ + }} \right\| = 1$. First we recall the definition of a Bussemann function.
\begin{definition}\cite{daicomparison,wei2007comparison,zhu1997comparison} Let $\gamma :\left[ {0, + \infty } \right] \to M$ be a ray, the Bussemann function ${b^\gamma }$ associated to $\gamma $ is defined as ${b^\gamma }(x): = \mathop {\lim }\limits_{t \to \infty } \left( {t - d(x,\gamma (t))} \right)$.
\end{definition}
For the next step we should prove that ${\Delta _{A,{f^A}}}\left( {{b^\gamma }} \right) \ge 0$ in the barrier sense.
\begin{prop}\label{busemansuphar} If $Ric({\partial _r},A{\partial _r}) \ge 0$, then ${\Delta _{A,{f^A}}}\left( {{b^\gamma }} \right) \ge 0$ in the barrier sense.
\end{prop}
\begin{proof}For each point  $q \in M$ the family of  functions defined as ${h}(x) = t - d(x,\overline \gamma  (t)) + {b^\gamma }(q)$ are a lower barrier functions of ${b^\gamma }$  at the point $ q $, where $\overline \gamma  (t)$ is one of asymptotic ray to the ray $ \gamma $ at $ q $  \cite{daicomparison,wei2007comparison,zhu1997comparison}. So ${h}$ is smooth in this neighborhood ${U}$ of $ q $. Finally by Theorem \ref{Meancurvaturecomparison} and the restriction on the extended Ricci tensor,
\[{\Delta _{A,{f^A}}}{h}(q) =  - {\Delta _{A,{f^A}}}\left( {d(q,\bar \gamma (t))} \right) \ge  - {\delta _n}\left( {1 + \frac{{4K}}{{{\delta _n}\left( {n - 1} \right)}}} \right)\frac{1}{{d(q,\bar \gamma (t))}}.\]
Since $\mathop {\lim }\limits_{t \to \infty } d(q,\overline \gamma  (t)) = \infty$, for each $\varepsilon  > 0$ one can finds some $ t $  with ${\Delta _A}\left( {d(q,\overline \gamma  (t))} \right) \ge  - \varepsilon$ and this complete the proof.
\end{proof}
\begin{cor}\label{busemansmooth}If $Ric({\partial _r},A{\partial _r}) \ge 0$ and ${\gamma _ + }$ and ${\gamma _ - }$ be two rays derived from the line $\gamma $ and ${b^ + }$ and ${b^ - }$ denote their Bussemann functions, then
\begin{itemize}
\item[a)]$b^ +  + b ^ -  = 0$,
\item[b)]${\Delta _{A,{f^A}}}{b^ + } = {\Delta _{A,{f^A}}}{b^ - } = 0$ and the functions $b^ +$ and $ b ^ - $ are smooth.
\end{itemize}
\end{cor}
\begin{proof} By Proposition \ref{busemansuphar} and the restriction on the extended Ricci tensor, \linebreak we know ${\Delta _{A,{f^A}}}{b^ + },{\Delta _{A,{f^A}}}{b^ - } \ge 0$. So
\[{\Delta_{A,{f^A}}}\left( {{b^ + } + {b^ - }} \right) \ge 0.\]
By triangle inequality we know $\left( {{b^ + } + {b^ - }} \right)(\gamma (0)) = 0$ is the maximum value of the sub harmonic function ${b^ + } + {b^ - }$, so by Theorem \ref{Maximum principle} we have ${b^ + } + {b^ - } = 0$.
For part (b) by the first part ${\Delta _{A,{f^A}}}{b^ + } =  - {\Delta _{A,{f^A}}}{b^ - }$, also Proposition \ref{busemansuphar}
concludes ${\Delta _{A,{f^A}}}{b^ + } = {\Delta_{A,{f^A}}}{b^ - } = 0$ , so by Proposition \ref{regular}  ${b^ + },{b^ - }$ are smooth.
\end{proof}
\begin{cor}\label{unitbgr}  If $Ric({\partial _r},A{\partial _r}) \ge 0$, then $\left\| {\nabla {b^\gamma }} \right\|= 1$.
\end{cor}
\begin{proof}  By Corollary \ref{busemansmooth} the Busseman functions ${b^ + }$ are smooth, so by  \cite{cheeger1971splitting, daicomparison,fang2009two,wei2007comparison,zhu1997comparison} we have \[\left\| {\nabla {b^ + }} \right\| = 1.\]
\end{proof}

\begin{proof} [\textbf {Proof of Theorem \ref{Cheeger}}] By Corollary \ref{unitbgr} we have $\left\| {\nabla {b^ + }} \right\| = 1$, so ${\nabla _{\nabla {b^ + }}}\nabla {b^ + } = 0$, also $ A $ is a Codazzi tensor thus $ - \nabla {b^ + }.\nabla {b^ + }.Trace(B) + \left\langle {{\nabla _{\nabla {b^ + }}}div(B),\nabla {b^ + }} \right\rangle  = 0$. By applying Theorem \ref{bochner3} to the function ${b^ + }$,
\[0 = Trace\left( {A \circ hes{s^2}\left( { {b^ + }} \right)} \right) + \nabla {b^ + }.({\Delta _A}{b^ + }) - \sum\nolimits_i {\left\langle {{T^{\left( {{\nabla _{\nabla {b^ + }}}A} \right)}}(\nabla {b^ + },{e_i}),{e_i}} \right\rangle  + Ric(\nabla {b^ + },A\nabla {b^ + })} \]
By Corollary \ref{busemansmooth} ${\Delta _{A,{f^A}}}{b^ + } = 0$, so ${\Delta _A}{b^ + } = \left\langle {\nabla {f^A},\nabla {b^ + }} \right\rangle $ and
\[\nabla {b^ + }.({\Delta _A}{b^ + }) = \nabla {b^ + }.\left\langle {\nabla {f^A},\nabla {b^ + }} \right\rangle  = Hess{f^A}\left( {\nabla {b^ + },\nabla {b^ + }} \right).\]
By Definition \ref{Bakryidea} we know
\[Hess{f^A}\left( {\nabla {b^ + },\nabla {b^ + }} \right) - \sum\nolimits_i {\left\langle {{T^{\left( {{\nabla _{\nabla {b^ + }}}A} \right)}}(\nabla {b^ + },{e_i}),{e_i}} \right\rangle }  \ge 0,\]
so
\[0 \ge Trace\left( {A \circ hes{s^2}\left( { {b^ + }} \right)} \right) + Ric(\nabla {b^ + },A\nabla {b^ + }).\]
The condition on the extended Ricci tensor concludes $Trace\left( {A \circ hes{s^2}({b^ + })} \right)=0$. We know $A$ is positive definite ( note that $A$ is invertible), so $hes{s^2}({b^ + }) \equiv 0$. Consequently ${\nabla {b^ + }}$ is a Killing vector field and its flow is an isometry. Also $\left\| {\nabla {b^ + }} \right\| = 1$, so the following function
\[\begin{array}{l}
 \psi :N \times\mathbb{R}  \to M \\
 \psi (x,t) = Fl^{\nabla {b^ + }}(x), \\
 \end{array}\]
 splits $M$ isometrically, where $N = \left\{ {x:{b^ + }(x) = 0} \right\}$.
 Also $\bar A$ is a Codazzi tensor, to see this note that,
\begin{eqnarray*}
\nabla _X^N\left( {\bar AY} \right) &=& \nabla _X^{}\left( {\bar AY} \right) = \nabla _X^{}\left( {AY - \left\langle {{\partial _t},AY} \right\rangle {\partial _t}} \right) = \nabla _X^{}\left( {AY} \right) - \left\langle {{\partial _t},\nabla _X^{}\left( {AY} \right)} \right\rangle {\partial _t} \\&=& pro{j_{\partial _t^ \bot }}\left( {\nabla _X^{}\left( {AY} \right)} \right),
\end{eqnarray*}
 and
 \[\bar A\left( {\nabla _X^NY} \right) = \bar A\left( {\nabla _X^{}Y} \right) = pro{j_{\partial _t^ \bot }}\left( {\nabla _X^{}\left( {AY} \right)} \right).\]
 So
 \[\left( {\nabla _X^N\bar A} \right)Y = \nabla _X^N\left( {\bar AY} \right) - \bar A\left( {\nabla _X^NY} \right) = pro{j_{\partial _t^ \bot }}\left( {\left( {\nabla _X^{}A} \right)Y} \right) = pro{j_{\partial _t^ \bot }}\left( {\left( {\nabla _Y^{}A} \right)X} \right) = \left( {\nabla _Y^N\bar A} \right)X.\]

The last part is clear by properties of Ricci tensor.
\end{proof}
Now we extend some famous consequences of Cheeger-Gromoll splitting theorem by some restrictions on the extended Ricci tensor. Similar to \cite{wei2007comparison} and \cite{zhu1997comparison} and lifting the tensor $Ric\left( {X,AX} \right)$  to $Ric\left( {X,\bar AX} \right)$ on the universal covering space of $ M $ as same as in the proof of Theorem \ref{Meyer's theorem} the following results are obtained .
\begin{theorem} If $M$ is a compact Riemannian manifold with $Ric\left( {X,AX} \right) \ge 0$ for any vector field $ X $, then $M$ is finitely covered by the universal covering ${N^{\dim M - k}} \times {T^k}$, where $N$ is compact and simply connected and $T^k$ is the flat torus of dimension $k$.
\end{theorem}

\begin{cor} Let $M$ be compact with $Ric\left( {X,AX} \right) \ge 0$ for any vector field $ X $, then
\begin{itemize}
\item[a)]$\,{b_1}(M) \le n$.
\item[b)] ${\pi _1}(M)$ has a free abelian subgroup of finite index of rank $\le n$.
\item[c)] If at one point $Ric\left( {X,AX} \right)>0$ for any non zero vector  $ X $, then ${\pi _1}(M)$ is finite.
\end{itemize}
\end{cor}
Also for non-compact manifolds and $Ric\left( {X,AX} \right)>0$ for any non zero vector  $ X $, the splitting theorem get the following extension of Sormani's theorem \cite{sormani2001loops} and a result similar to Theorem 6.8 of \cite{wei2007comparison}.
\begin{theorem}Let $M$ is a complete and non-compact manifold and $Ric\left( {X,AX} \right)>0$ for any non zero vector  $ X $, then
\begin{itemize}
\item[a)]$M$ has only one end.
\item[b)]$M$ has the loops to infinity property. In particular, if $M$ is simply connected at infinity then $M$ is simply
connected.
\end{itemize}
\end{theorem}
\section{excess functions and applications}\label{sec8}
Excess functions are important in the study of topology of manifolds, so the upper bounds of them are interesting. Let $p,q \in M$ be two point, we recall that the excess function $ {e_{p,q}}(x) $  is defined as
\begin{equation}
{e_{p,q}}(x): = d(p,x) + d(q,x) - d(p,q).
\label{defexcess}
\end{equation}

As usual, to estimate of excess function we need an extension of Abresch-Gromoll quantitative maximal principle.  The proof is an adaptation of the exposition contained in Abresch-Gromoll \cite {abresch1990complete} or Cheeger’s \cite{cheeger2004degeneration} works. At first we recall the following definition.
\begin{definition} The dilation of a function $ f $ is denoted by $ dil(f) $ and is defined as
\[dil(f) = \mathop {\min }\limits_{x,y} \frac{{\left| {f(x) - f(y)} \right|}}{{d(x,y)}}.\]
\end{definition}
\begin{prop}( Quantitative Maximal Principle )\label{Quantitative Maximal Principle} Let $ U:B(y,R + \eta ) \to \mathbb{R} $ is a Lipschitz function on $ M $. For $H \le 0$, Assume $\left( {n - 1} \right){\delta _n}H \le Ric\left( {{\partial _r},A{\partial _r}} \right)$, $\left| {{f^A}} \right| \le K$ and
\begin{itemize}
\item[a)] $U \ge 0$,
\item[b)] $dil(U) \le a$, $U({y_0}) = 0$, where ${y_0}$ is a point in $\overline {B(y,R)} $.
\item[b)]${\Delta _{A,{f^A}}}(U) \le b$ in the barrier sense,
\end{itemize}
then $U(y) \le ac +G(c)$ for all $0 < c < R $ where $G(r(x))$ is the unique function on $M_H^n$ such that
\begin{itemize}
\item[1)]$G(r) > 0$ for $0 < r < R$,
\item[2)]$G'(r) < 0$ for $0 < r < R$,
\item[3)]$G(R) = 0$,
\item[4)]$\left( {{\delta _1} - {\delta _n}\left( {1 + \frac{{4K}}{{{\delta _n}(n - 1)}}} \right)} \right)G'' + {\delta _n}\left( {1 + \frac{{4K}}{{{\delta _n}(n - 1)}}} \right){\Delta _H}G = b$, where
${\Delta _H}$ is the Laplace operator on the model space $M_H^n$.
\end{itemize}
\end{prop}
\begin{proof}At first, for late use we construct $ G $ explicitly. Since ${\Delta _H} = \frac{{{\partial ^2}}}{{\partial {r^2}}} + {m_H}(r)\frac{\partial }{{\partial r}} + \tilde \Delta $. Its sufficient to solve the ODE
\[\left( {{\delta _1} - {\delta _n}\left( {1 + \frac{{4K}}{{{\delta _n}(n - 1)}}} \right)} \right)G'' + {\delta _n}\left( {1 + \frac{{4K}}{{{\delta _n}(n - 1)}}} \right)\left( {G'' + {m_H}(r)G'} \right) = b,\]
or equivalently,
\begin{equation}
{\delta _1}G'' + {\delta _n}\left( {1 + \frac{{4K}}{{{\delta _n}(n - 1)}}} \right){m_H}(r)G' = b.
\label{AGM1}
\end{equation}
For $ H=0 $ we know ${m_H}(r) = \frac{{n - 1}}{r}$, so by (\ref{AGM1}) we have
\[{\delta _1}G'' + {\delta _n}\left( {1 + \frac{{4K}}{{{\delta _n}(n - 1)}}} \right)\frac{{n - 1}}{r}G' = b,\]
 or equivalently
 \[{\delta _1}G''{r^2} + \left( {n - 1} \right){\delta _n}\left( {1 + \frac{{4K}}{{{\delta _n}(n - 1)}}} \right)G'r = b{r^2},\]
which is an Euler-type ODE.
For $n \ge 3$, the solutions of this ODE are as,
\[G = \frac{b}{{2C}}{r^2} + {c_1} + {c_2}{r^D},\]
where
\[C = {\delta _1} + \left( {n - 1} \right){\delta _n}\left( {1 + \frac{{4K}}{{{\delta _n}(n - 1)}}} \right)\,\,\,\,\,\,and\,\,\,\,\,D = 1 - \left( {n - 1} \right)\frac{{{\delta _n}}}{{{\delta _1}}}\left( {1 + \frac{{4K}}{{{\delta _n}(n - 1)}}} \right).\]
 Now $G(R) = 0$, gives
\[\frac{b}{{2C}}{R^2} + {c_1} + {c_2}{R^D} = 0.\]
 By assumption $G'(r) < 0$, so for $0 < r < R$ one should have,
\[\frac{b}{C}r + D{c_2}{r^{D - 1}} \le 0.\]
 Thus, ${c_2} \ge  - \frac{b}{{CD}}{R^{ - D + 2}}$. So for $ H=0 $, we have,
 \[G(r) = \frac{b}{{2C}}\left( {{r^2} + \left( { - 1 + \frac{2}{D}} \right){R^2} - \frac{1}{D}{R^{ - D + 2}}{r^D}} \right).\]
 For $ H<0 $, we have,${m_H}(r) = \left( {n - 1} \right)\sqrt { - H} \frac{{\cosh (\sqrt { - H} r)}}{{\sinh (\sqrt { - H} r)}}$. So by (\ref{AGM1}) we have,
\[{\delta _1}G'' + {\delta _n}\left( {n + \frac{{4K}}{{{\delta _n}}} - 1} \right)\sqrt { - H} \frac{{\cosh (\sqrt { - H} )}}{{\sinh (\sqrt { - H} r)}}G' = b.\]
 Thus,
\[G(r) = \frac{b}{{{\delta _1}}}\int_r^R {\int_r^t {{{\left( {\frac{{\sinh (\sqrt { - H} t)}}{{\sinh (\sqrt { - H} s)}}} \right)}^{\frac{{{\delta _n}}}{{{\delta _1}}}\left( {1 + \frac{{4K}}{{{\delta _n}(n - 1)}}} \right)}}dsdt} } .\]
 Now we return to prove the result of the theorem. By conditions on $ G $ and Lemma \ref{weakinequalityba} \[{\Delta _{A,{f^A}}}G \ge \left( {{\delta _1} - {\delta _n}\left( {1 + \frac{{4K}}{{{\delta _n}(n - 1)}}} \right)} \right)G'' + {\delta _n}\left( {1 + \frac{{4K}}{{{\delta _n}(n - 1)}}} \right){\Delta _H}G.\]
 Define $V: = G -U$, so
\begin{eqnarray*}
   {{\Delta _{A,{f^A}}}V}  & =&{  {\Delta _{A,{f^A}}}G - {\Delta _{A,{f^A}}}U} \hfill  \\
  &  \ge& {\left( {{\delta _1} - {\delta _n}\left( {1 + \frac{{4K}}{{{\delta _n}(n - 1)}}} \right)} \right)G'' + {\delta _n}\left( {1 + \frac{{4K}}{{{\delta _n}(n - 1)}}} \right){\Delta _H}G - {\Delta _{A,{f^A}}}U \ge 0.}
\end{eqnarray*}
By maximum principle theorem \ref{Maximum principle}, the function $ V $ on $A(y,c,R) = \left\{ {x:c < d(y,x) < R} \right\}$ takes its maximum on $\partial B(y,c) \cup \partial B(y,R)$. But ${\left. V \right|_{\partial B(y,R)}} \le 0$ and
 $\,V({y_0}) \ge 0$, so if ${y_0} \in A(y,c,R)$, then there exists some ${y_1} \in \partial B(y,c)$ such that $\,V({y_1}) \ge \,V({y_0}) \ge 0$. Since
 \[U(y) - U({y_1}) \le a\,d(y,{y_1}) = ac\]
 and
\[0 \le V({y_1}) = G({y_1}) - U({y_1}).\]
So
\[U(y) \le ac + U({y_1}) = ac + \left( {G({y_1}) - V({y_1})} \right) \le ac + G(c).\]
Also, if ${y_0} \in B(y,c)$ then
\[U(y) = U(y) - U({y_0}) \le a{\mkern 1mu} d(y,{y_0}) \le ac \le ac + G(c).\]
\end{proof}
By Proposition \ref{Quantitative Maximal Principle} we get the following upper estimate for the excess function. To do this we define the height function $h(x): = dist(x,\gamma )$, where $dist(x,\gamma )$ is any fixed minimal geodesic from $ p $ to $ q $.
\begin{theorem}\label{excessfunctionestimate} Let $Ric\left( {{\partial _r},A{\partial _r}} \right) \ge 0$, $\left| {{f^A}} \right| \le K$ and $h(x) \le \min \{ d(p,x),d(q,x)\}$  then
\[{e_{p,q}}(x) \le 2\left( {\frac{{{\delta _n}(n - 1) + 4K + {\delta _1}}}{{{\delta _n}(n - 1) + 4K - {\delta _1}}}} \right){\left( {\frac{1}{2}C{h^{\frac{{{\delta _1} + {\delta _n}(n - 1) + 4K}}{{{\delta _1}}}}}} \right)^{\frac{{{\delta _1}}}{{{\delta _n}(n - 1) + 4K}}}},\]
where
\[C = \frac{{{\delta _n}(n - 1) + 4K}}{{2(n - 1)\left( {{\delta _1} + {\delta _n}(n - 1) + 4K} \right)}}\left( {\frac{1}{{d(p,x) - h(x)}} + \frac{1}{{d(q,x) - h(x)}}} \right).\]
\end{theorem}
\begin{proof} As original one, we use the extended Abresch and Gromoll's Quantitative Maximal Principle. Note that
$dil({e_{p,q}}) \le 2$. If we choose $  R=h(x) $, then for any $y \in B(x,R)$ we have
\begin {eqnarray*}
   {{\Delta _A}\left( {{e_{p,q}}(y)} \right)}  &  \le&  {{\delta _n}\left( {1 + \frac{{4K}}{{{\delta _n}(n - 1)}}} \right)\left( {\frac{1}{{d(p,y)}} + \frac{1}{{d(q,y)}}} \right)}   \\
  & \le&  {{\delta _n}\left( {1 + \frac{{4K}}{{{\delta _n}(n - 1)}}} \right)\left( {\frac{1}{{d(p,x) - h(x)}} + \frac{1}{{d(q,x) - h(x)}}} \right).} \end{eqnarray*}
So by choosing $  R=h(x) $ and $b: = {\delta _n}\left( {1 + \frac{{4K}}{{{\delta _n}(n - 1)}}} \right)\left( {\frac{1}{{d(p,x) - h(x)}} + \frac{1}{{d(q,x) - h(x)}}} \right)$, the conditions of Proposition \ref{Quantitative Maximal Principle} are satisfied. So
\[{e_{p,q}}(x) \le \mathop {\min }\limits_{0 \le r \le R} \left( {2r + G(r)} \right).\]
The function ${2r + G(r)}$, for $0 < r < R$ is convex, so its minimum is assumed at the unique $0 < {r_0} < R$, where $2 + G'({r_0})=0$, from this we conclude,
\begin{equation}
2r_0^{1 - D} + \frac{b}{{2C}}\left( {2r_0^{2 - D} - {R^{2 - D}}} \right) = 0,
\label{excess01}
\end{equation}
where $ C,D $ are defined in Proposition \ref{Quantitative Maximal Principle}. By (\ref {excess01}) we get,
\[r_0^{} \le {\left( {\frac{b}{{4C}}{R^{2 - D}}} \right)^{\frac{1}{{1 - D}}}}.\]
Consequently,
\begin{eqnarray*}
   {{e_{p,q}}(x)} \hfill &  \le& {2{r_0} + G({r_0}) = \left( {1 - \frac{2}{D}} \right)\left[ {2{r_0} + \frac{b}{C}\left( {r_0^2 - {R^2}} \right)} \right] \le 2\left( {1 - \frac{2}{D}} \right){r_0}}   \\
 &  \le& {2\left( {1 - \frac{2}{D}} \right){{\left( {\frac{b}{{4C}}{R^{2 - D}}} \right)}^{\frac{1}{{1 - D}}}}.}
\end{eqnarray*}
\end{proof}
So applying the estimate of excess function Theorem \ref{excessfunctionestimate} gets an extension of theorems of Abresch-Gromoll \cite{abresch1990complete} and Sormani \cite{sormani1998nonnegative} as follows.
\begin{theorem}Let  $ M $ be a complete non compact manifold with $Ric\left( {{\partial _r},A{\partial _r}} \right) \ge 0$ then,
\begin{itemize}
\item[a)]If $ M $ has bounded diameter growth and sectional curvature bounded below then it has finite type topology, i.e, it is homeomorphic to the interior of a compact manifold with boundary.
\item[b)] If it has sub-linear diameter growth, then its fundamental group is finitely generated.
\end{itemize}
\end{theorem}
\section{Numbers of ends}\label{sec9}
In this section we get an estimate on the number of ends of a complete Riemannian manifold $ M $ by some restrictions on $Ri{c_A}$. The approach is similar to \cite{cai1991ends}. In fact M.Cai invented this method to estimate the number of ends of a noncompact manifold which its Ricci curvature is non-negative out-side of a compact set by means of a lower bound of the Ricci curvature in the compact set and the diameter of the set. Recently  J. Wu  applied this method to get an upper estimate by similar conditions on the Bakry-Emery Ricci tensor $Ri{c_f}$ and some conditions on the energy function \cite{wu2016counting}. Similarly we use the method to get our estimate.
\\ First we recall the definition of ends of manifolds.
\begin{definition}\cite{cai1991ends} Let ${\gamma _1},{\gamma _2}$ be two rays starting from a fixed point $p \in M$. We say ${\gamma _1},{\gamma _2}$ are co-final if for each $R > 0$ and any $t \ge R$, ${\gamma _1}(t)$ and ${\gamma _2}(t)$ are in the same component of $M\backslash B(x_0,r)$. Each equivalence class of co-final rays is called an end of $ M $. The end included the ray $\gamma $ is noted by $\left[ \gamma  \right]$.
\end{definition}
To get the estimate, we extend the following lemmas for the extended Ricci tensor $Ri{c_A}$. The proofs of the following lemmas are similar to the corresponding for Ricci \cite{cai1991ends} or weighted Ricci tensor \cite{wu2016counting}, so they are omitted.
\begin{lem} Let $ N $ be a $\delta  -$tubular neighborhood of a line $ \gamma $. Suppose that from every point $p  $ in N, there are asymptotic rays to ${\gamma ^ \pm }$ such that $Ri{c_{Trace(A)}}(X,AX) \ge 0$ on both asymptotic rays. Then through every point in N, there is a line $\alpha $ which if it is parameterized properly, then satisfies
\[b_\gamma ^ + ({\alpha ^ + }(t)) = t\,\,\,\,\,\,\,\,and\,\,\,\,\,\,b_\gamma ^ - ({\alpha ^ - }(t)) = t.\]
\end{lem}
\begin{lem} With the same assumptions as in Theorem \ref{endofmanifold}, $ M $ can not admit a line $\gamma $ with the following property
\[d(\gamma (t),B(x_0,r)) \ge \left| t \right| + 2R\,\,\,\,\,\,\,for\,all\,\,t.\]
\end{lem}
Similar to \cite{cai1991ends} the following Proposition can be obtained.
\begin{prop}\label{en1} With the same assumption as in theorem \ref{endofmanifold}, if $\left[ {{\gamma _1}} \right]$ and $\left[ {{\gamma _2}} \right]$ are two different ends of $ M $, then $d({\gamma _1}(4R),{\gamma _2}(4R)) > 2R$.
\end{prop}

\begin{proof}[\textbf{Proof of Theorem \ref{endofmanifold}}] Let $\left[ {{\gamma _1}} \right],...,\left[ {{\gamma _k}} \right]$ be $ k $ disjoint ends of $ M $ where ${\gamma _1},...,{\gamma _k}$ are rays from the fixed point $ p $. Let $\left\{ {{p_j}} \right\}_{j = 1}^L$ be the maximal set of points on $\partial B\left( {p,4R} \right)$ such that balls $B\left( {{p_j},\frac{R}{2}} \right)$ are disjoint. As mentioned in \cite{wu2016counting,cai1991ends} Proposition \ref{en1} implies that $k \le L$. \\
By considering
\[B({p_j},\frac{R}{2}) \subseteq B(p,\frac{{9R}}{2}) \subseteq B({p_j},\frac{{17R}}{2}),\]
we get
\[N(A,M,R) \le \frac{{vol_{{h^A}}^A\left( {B({p_j},\frac{{17R}}{2})} \right)}}{{vol_{{h^A}}^A\left( {B({p_j},\frac{R}{2})} \right)}},\]
and with the extended volume comparison Theorem \ref{extended volume comparison}, the following estimate is obtained,
\[N(A,M,R) \le \frac{{vol_{{h^A}}^A\left( {B({p_j},\frac{{17R}}{2})} \right)}}{{vol_{{h^A}}^A\left( {B({p_j},\frac{R}{2})} \right)}} \le {\left( {\frac{{1 - \bar C\varepsilon }}{{1 - 2\bar C\varepsilon }}} \right)^p}\frac{{vol_{ - H}^{m'}B(R)}}{{vol_{ - H}^{m'}B(r)}},\]
where
\[\begin{array}{*{20}{c}}
   {m' = \left[ {\frac{{{\delta _n}(n - 1) + 4({K_1} + {{K'}_1})}}{{{\delta _1}}}} \right] + 2,}  \\
   {{K_1} = \mathop {\sup }\limits_{x \in B({p_j},17R/2)} \left| {{f^A}\left( x \right)} \right|,}  \\
   {{{K'}_1}: = \mathop {\sup }\limits_{x \in B({p_j},17R/2)} \left| {Trace(A)\left( x \right)} \right|.}  \\
\end{array}\]
But for all $ j $, we have $B({p_j},17R/2) \subseteq B({x_0},25R/2)$ and $\left( {\frac{{1 - \bar C\varepsilon }}{{1 - 2\bar C\varepsilon }}} \right)$ is non decreasing with respect to $ R $, so the resut follows by the following inequality,
\[\frac{{\int_0^{\alpha r} {{{\sinh }^{m - 1}}(\beta t)dt} }}{{\int_0^r {{{\sinh }^{m - 1}}(\beta t )dt} }} \le \frac{{2m}}{{m - 1}}{\left( {\beta r} \right)^{ - m}}\exp \left( {\alpha \left( {m - 1} \right)\beta r} \right).\]
\end{proof}

\section{Applications in the study of geometry of hypersurfaces}\label{sec10}
In this section we give an application of the extended Ricci tensor for the study of Riemannian hyper surfaces, to do this we apply the tensor field $\frac{1}{{\delta _n^2}}{A^2}$ and show for minimal hypersurfaces the extended Ricci tensor $\frac{1}{{\delta _n^2}}Ric\left( {X,{A^2}X} \right)$ is greater than the Ricci tensor of the hypersurface. At first We compute the corresponding parameters for the tensor field $ A^2 $, when $ A $ is a codazzi tensor and get a similar Ricatti inequality as we get for Codazzi tensors.
\begin{lem}\label{application1} Let $ A $ be a Codazzi tensor, then
\[\sum\nolimits_i {{e_i}.\left\langle {\nabla u,{T^{{A^2}}}({e_i},\nabla u)} \right\rangle } {\rm{ }} = \left\langle {div\left( {\left( {{\nabla _{\nabla u}}A} \right)A - A\left( {{\nabla _{\nabla u}}A} \right)} \right),\nabla u} \right\rangle. \]
\end{lem}
\begin{proof} At first note that,
\[{T^{{A^2}}}\left( {X,Y} \right) = \left( {{\nabla _X}{A^2}} \right)Y - \left( {{\nabla _Y}{A^2}} \right)X = \left( {{\nabla _X}A} \right)AY - \left( {{\nabla _Y}A} \right)AX.\]
So,
\begin{eqnarray*}
   {\sum\nolimits_i {{e_i}.\left\langle {\nabla u,{T^{{A^2}}}({e_i},\nabla u)} \right\rangle } } &=&\sum\nolimits_i {{e_i}.\left\langle {\nabla u,\left( {{\nabla _{{e_i}}}A} \right)A\nabla u - \left( {{\nabla _{\nabla u}}A} \right)A{e_i}} \right\rangle }
\\&=&\sum\nolimits_i {{e_i}.\left( {\left\langle {\left( {{\nabla _{{e_i}}}A} \right)\nabla u,A\nabla u} \right\rangle  - \left\langle {\nabla u,\left( {{\nabla _{\nabla u}}A} \right)A{e_i}} \right\rangle } \right)}
\\&=&\sum\nolimits_i {{e_i}.\left( {\left\langle {{e_i},\left( {{\nabla _{\nabla u}}A} \right)A\nabla u} \right\rangle  - \left\langle {A\left( {{\nabla _{\nabla u}}A} \right)\nabla u,{e_i}} \right\rangle } \right)}
\\&=&\sum\nolimits_i {{e_i}.\left( {\left\langle {\left( {{\nabla _{\nabla u}}A} \right)A\nabla u - A\left( {{\nabla _{\nabla u}}A} \right)\nabla u,{e_i}} \right\rangle } \right)}
\\&=&div\left( {\left( {\left( {{\nabla _{\nabla u}}A} \right)A - A\left( {{\nabla _{\nabla u}}A} \right)} \right)\nabla u} \right).
\end{eqnarray*}
Note,
\begin{eqnarray*}
  {div\left( {\left( {\left( {{\nabla _{\nabla u}}A} \right)A - A\left( {{\nabla _{\nabla u}}A} \right)} \right)\nabla u} \right)}&=&\left\langle {div\left( {\left( {{\nabla _{\nabla u}}A} \right)A - A\left( {{\nabla _{\nabla u}}A} \right)} \right),\nabla u} \right\rangle
\\&&+ {\Delta _{\left( {{\nabla _{\nabla u}}A} \right)A - A\left( {{\nabla _{\nabla u}}A} \right)}}u.
\end{eqnarray*}
and,
\begin{eqnarray*}
   {\sum\nolimits_i {\left\langle {hessu({e_i}),\left( {{\nabla _{\nabla u}}A} \right)A{e_i}} \right\rangle } }&=&\sum\nolimits_i {\left\langle {\left( {{\nabla _{\nabla u}}A} \right) \circ hessu({e_i}),A{e_i}} \right\rangle  = \sum\nolimits_i {\left\langle {\left( {{\nabla _{hessu({e_i})}}A} \right)\nabla u,A{e_i}} \right\rangle } }
\\&=&\sum\nolimits_i {\left\langle {\nabla u,\left( {{\nabla _{hessu({e_i})}}A} \right)A{e_i}} \right\rangle }  = \sum\nolimits_i {\left\langle {\nabla u,\left( {{\nabla _{A{e_i}}}A} \right)hessu({e_i})} \right\rangle }
\\&=&\sum\nolimits_i {\left\langle {\left( {{\nabla _{A{e_i}}}A} \right)\nabla u,hessu({e_i})} \right\rangle }  = \sum\nolimits_i {\left\langle {\left( {{\nabla _{\nabla u}}A} \right)A{e_i},hessu({e_i})} \right\rangle }
\\&=&\sum\nolimits_i {\left\langle {hessu \circ \left( {{\nabla _{\nabla u}}A} \right)A{e_i},{e_i}} \right\rangle }  = Trace\left( {hessu \circ \left( {{\nabla _{\nabla u}}A} \right)A} \right)
\\&=&Trace\left( {\left( {{\nabla _{\nabla u}}A} \right)A \circ hessu} \right) = \sum\nolimits_i {\left\langle {\left( {{\nabla _{\nabla u}}A} \right)A \circ hessu({e_i}),{e_i}} \right\rangle }
\\&=&\sum\nolimits_i {\left\langle {hessu({e_i}),A\left( {{\nabla _{\nabla u}}A} \right){e_i}} \right\rangle }.
\end{eqnarray*}
So,
\[{\Delta _{\left( {{\nabla _{\nabla u}}A} \right)A - A\left( {{\nabla _{\nabla u}}A} \right)}}u = 0,\]
and,
\[div\left( {\left( {\left( {{\nabla _{\nabla u}}A} \right)A - A\left( {{\nabla _{\nabla u}}A} \right)} \right)\nabla u} \right) = \left\langle {div\left( {\left( {{\nabla _{\nabla u}}A} \right)A - A\left( {{\nabla _{\nabla u}}A} \right)} \right),\nabla u}\right\rangle. \]
\end{proof}
Now we compute $\sum\nolimits_i {\left\langle {{T^{{\nabla _{\nabla u}}{A^2}}}({e_i},\nabla u),{e_i}} \right\rangle } $.
\begin{prop}\label{application2}Let ${\Sigma ^n} \subset {M^{n + 1}}(c)$ be a Riemannian hypersurface and the ambient manifold $ M $ has constant sectional curvature with the shape operator $ A $, such that ${\nabla ^*}{T^{{A^2}}} = 0$ , then the extended Bochner formula \ref{extendedbochner2} is written as (For the operator $ A^2 $)
\begin{eqnarray*}
   {\frac{1}{2}{L_{{A^2}}}({{\left| {\nabla u} \right|}^2})  }  &=& {\frac{1}{2}\left\langle {\nabla {{\left| {\nabla u} \right|}^2},div({A^2})} \right\rangle  + Trace\left( {{A^2} \circ hes{s^2}\left( u \right)} \right) + \left\langle {\nabla u,\nabla ({\Delta _{{A^2}}}u)} \right\rangle }   \\
 && { - \nabla u.\nabla u.Trace({A^2}) + \left\langle {{\nabla _{\nabla u}}div({A^2}),\nabla u} \right\rangle  + Ric(\nabla u,{A^2}\nabla u)}   \\
 && { - 2\sum\nolimits_i {\left\langle {{T^{{\nabla _{\nabla u}}A}}\left( {{e_i},\nabla u} \right),A{e_i}} \right\rangle }  - 2\left\langle {divA,\left( {{\nabla _{\nabla u}}A} \right)\nabla u} \right\rangle  + 2Trace\left( {{{\left( {{\nabla _{\nabla u}}A} \right)}^2}} \right)}
\end{eqnarray*}
\end{prop}
\begin{proof}We know
\[\sum\nolimits_i {\left\langle {{T^{{\nabla _{\nabla u}}{A^2}}}({e_i},\nabla u),{e_i}} \right\rangle }  = \sum\nolimits_i {\left\langle {{T^{{\nabla _{\nabla u}}{A^2}}}({e_i},\nabla u),{e_i}} \right\rangle }  = \sum\nolimits_i {\left\langle {{T^{A\left( {{\nabla _{\nabla u}}A} \right) + A\left( {{\nabla _{\nabla u}}A} \right)}}({e_i},\nabla u),{e_i}} \right\rangle } ,\]
By computation of $\sum\nolimits_i {\left\langle {{T^{A\left( {{\nabla _{\nabla u}}A} \right)}}({e_i},\nabla u),{e_i}} \right\rangle }$, we have
\begin{eqnarray*}
   {\sum\nolimits_i {\left\langle {{T^{A\left( {{\nabla _{\nabla u}}A} \right)}}({e_i},\nabla u),{e_i}} \right\rangle } }  & = &{ \sum\nolimits_i {\left\langle {\left( {{\nabla _{{e_i}}}\left( {A\left( {{\nabla _{\nabla u}}A} \right)} \right)} \right)\nabla u - \left( {{\nabla _{\nabla u}}\left( {A\left( {{\nabla _{\nabla u}}A} \right)} \right)} \right){e_i},{e_i}} \right\rangle } }   \\
  & = &{ \sum\nolimits_i {\left\langle {\left( {{\nabla _{{e_i}}}A} \right)\left( {{\nabla _{\nabla u}}A} \right)\nabla u,{e_i}} \right\rangle }  + \sum\nolimits_i {\left\langle {A\left( {{\nabla _{{e_i}}}{\nabla _{\nabla u}}A} \right)\nabla u,{e_i}} \right\rangle } }   \\
 && {\,\,\,\,\, - \sum\nolimits_i {\left\langle {{{\left( {{\nabla _{\nabla u}}A} \right)}^2}{e_i},{e_i}} \right\rangle }  - \sum\nolimits_i {\left\langle {A\left( {\nabla _{\nabla u}^2A} \right){e_i},{e_i}} \right\rangle } } \hfill  \\
    & =& { \left\langle {\left( {{\nabla _{\nabla u}}A} \right)\nabla u,divA} \right\rangle  + \sum\nolimits_i {\left\langle {A\left( {{\nabla _{{e_i}}}{\nabla _{\nabla u}}A} \right)\nabla u,{e_i}} \right\rangle } } \hfill  \\
   & & {\,\,\,\, - Trace\left( {{{\left( {{\nabla _{\nabla u}}A} \right)}^2}} \right) - Trace\left( {A\left( {\nabla _{\nabla u}^2A} \right)} \right).}
\end{eqnarray*}
Also
\begin{eqnarray*}
   {\sum\nolimits_i {\left\langle {{T^{\left( {{\nabla _{\nabla u}}A} \right)A}}({e_i},\nabla u),{e_i}} \right\rangle } }  & =& { \sum\nolimits_i {\left\langle {\left( {{\nabla _{{e_i}}}\left( {\left( {{\nabla _{\nabla u}}A} \right)A} \right)} \right)\nabla u - \left( {{\nabla _{\nabla u}}\left( {\left( {{\nabla _{\nabla u}}A} \right)A} \right)} \right){e_i},{e_i}} \right\rangle } }   \\
   &  =& {\sum\nolimits_i {\left\langle {\left( {{\nabla _{{e_i}}}{\nabla _{\nabla u}}A} \right)A\nabla u,{e_i}} \right\rangle }  + \sum\nolimits_i {\left\langle {\left( {{\nabla _{\nabla u}}A} \right)\left( {{\nabla _{{e_i}}}A} \right)\nabla u,{e_i}} \right\rangle } }   \\
  && {\,\,\,\, - \sum\nolimits_i {\left\langle {\left( {\nabla _{\nabla u}^2A} \right)A{e_i},{e_i}} \right\rangle  - \sum\nolimits_i {\left\langle {{{\left( {{\nabla _{\nabla u}}A} \right)}^2}{e_i},{e_i}} \right\rangle } } }   \\
   &  =&{ \sum\nolimits_i {\left\langle {\left( {{\nabla _{{e_i}}}{\nabla _{\nabla u}}A} \right)A\nabla u,{e_i}} \right\rangle }  + Trace\left( {{{\left( {{\nabla _{\nabla u}}A} \right)}^2}} \right)\,} \hfill  \\
  && {\,\,\,\, - Trace\left( {{{\left( {{\nabla _{\nabla u}}A} \right)}^2}} \right) - Trace\left( {\left( {\nabla _{\nabla u}^2A} \right)A} \right)} \hfill  \\
  &=& {  \sum\nolimits_i {\left\langle {\left( {{\nabla _{{e_i}}}{\nabla _{\nabla u}}A} \right)A\nabla u,{e_i}} \right\rangle }  - Trace\left( {\left( {\nabla _{\nabla u}^2A} \right)A} \right).}
\end{eqnarray*}
Consequently,
\begin{eqnarray*}
   {\sum\nolimits_i {\left\langle {{T^{{\nabla _{\nabla u}}{A^2}}}({e_i},\nabla u),{e_i}} \right\rangle } }  &  =&{ \sum\nolimits_i {\left\langle {{T^{A\left( {{\nabla _{\nabla u}}A} \right)}}({e_i},\nabla u),{e_i}} \right\rangle }  + \sum\nolimits_i {\left\langle {{T^{\left( {{\nabla _{\nabla u}}A} \right)A}}({e_i},\nabla u),{e_i}} \right\rangle } }   \\
 &= & { \left\langle {\left( {{\nabla _{\nabla u}}A} \right)\nabla u,divA} \right\rangle  + \sum\nolimits_i {\left\langle {A\left( {{\nabla _{{e_i}}}{\nabla _{\nabla u}}A} \right)\nabla u,{e_i}} \right\rangle } }   \\
   & & {\,\,\,\, - Trace\left( {{{\left( {{\nabla _{\nabla u}}A} \right)}^2}} \right) - Trace\left( {A\left( {\nabla _{\nabla u}^2A} \right)} \right)} \hfill  \\
 && {\,\,\,\, + \sum\nolimits_i {\left\langle {\left( {{\nabla _{{e_i}}}{\nabla _{\nabla u}}A} \right)A\nabla u,{e_i}} \right\rangle }  - Trace\left( {\left( {\nabla _{\nabla u}^2A} \right)A} \right).}
\end{eqnarray*}
By computation of $\left\langle {div\left( {\left( {{\nabla _{\nabla u}}A} \right)A - A\left( {{\nabla _{\nabla u}}A} \right)} \right),\nabla u} \right\rangle$ we have,
\begin{eqnarray*}
   &&{\left\langle {div\left( {\left( {{\nabla _{\nabla u}}A} \right)A - A\left( {{\nabla _{\nabla u}}A} \right)} \right),\nabla u} \right\rangle } \\ & = &{ \sum\nolimits_i {\left\langle {\left( {{\nabla _{{e_i}}}\left( {\left( {{\nabla _{\nabla u}}A} \right)A - A\left( {{\nabla _{\nabla u}}A} \right)} \right)} \right){e_i},\nabla u} \right\rangle } }   \\
 &  =&{ \sum\nolimits_i {\left\langle {\left( {{\nabla _{{e_i}}}{\nabla _{\nabla u}}A} \right)A{e_i},\nabla u} \right\rangle }  + \sum\nolimits_i {\left\langle {\left( {{\nabla _{\nabla u}}A} \right)\left( {{\nabla _{{e_i}}}A} \right){e_i},\nabla u} \right\rangle } }   \\
& & {\,\,\,\,\, - \sum\nolimits_i {\left\langle {\left( {{\nabla _{{e_i}}}A} \right)\left( {{\nabla _{\nabla u}}A} \right){e_i},\nabla u} \right\rangle }  - \sum\nolimits_i {\left\langle {\left( {A\left( {{\nabla _{{e_i}}}{\nabla _{\nabla u}}A} \right)} \right){e_i},\nabla u} \right\rangle } } \hfill  \\
   & =& {  \sum\nolimits_i {\left\langle {\left( {{\nabla _{{e_i}}}{\nabla _{\nabla u}}A} \right)A{e_i},\nabla u} \right\rangle }  + \left\langle {divA,\left( {{\nabla _{\nabla u}}A} \right)\nabla u} \right\rangle } \hfill  \\
  && {\,\,\,\,\, - Trace\left( {{{\left( {{\nabla _{\nabla u}}A} \right)}^2}} \right) - \sum\nolimits_i {\left\langle {\left( {A\left( {{\nabla _{{e_i}}}{\nabla _{\nabla u}}A} \right)} \right){e_i},\nabla u} \right\rangle } .}
\end{eqnarray*}
And finally we have
\begin{eqnarray*}
   &&{\sum\nolimits_i {\left\langle {{T^{{\nabla _{\nabla u}}{A^2}}}({e_i},\nabla u),{e_i}} \right\rangle }  + \sum\nolimits_i {{e_i}.\left\langle {{T^{{A^2}}}({e_i},\nabla u),{e_i}} \right\rangle } }   \\
& =&  { \sum\nolimits_i {\left\langle {{T^{{\nabla _{\nabla u}}{A^2}}}({e_i},\nabla u),{e_i}} \right\rangle }  + \left\langle {div\left( {\left( {{\nabla _{\nabla u}}A} \right)A - A\left( {{\nabla _{\nabla u}}A} \right)} \right),\nabla u} \right\rangle }   \\
 &=&  {  2\sum\nolimits_i {\left\langle {\left( {{\nabla _{{e_i}}}{\nabla _{\nabla u}}A} \right)A{e_i},\nabla u} \right\rangle }  + 2\left\langle {divA,\left( {{\nabla _{\nabla u}}A} \right)\nabla u} \right\rangle }   \\
   &&{ - 2Trace\left( {{{\left( {{\nabla _{\nabla u}}A} \right)}^2}} \right) - 2Trace\left( {A\left( {\nabla _{\nabla u}^2A} \right)} \right)} \hfill  \\
  &=& { 2\sum\nolimits_i {\left\langle {A{e_i},{T^{{\nabla _{\nabla u}}A}}\left( {{e_i},\nabla u} \right)} \right\rangle }  + 2\left\langle {divA,\left( {{\nabla _{\nabla u}}A} \right)\nabla u} \right\rangle  - 2Trace\left( {{{\left( {{\nabla _{\nabla u}}A} \right)}^2}} \right).}
\end{eqnarray*}
And from the Theorem \ref{extendedbochner2} the result follows.
\end{proof}
\begin{theorem}\label{application3}Let ${\Sigma ^n} \subset {M^{n + 1}}(c)$ be a Riemannian hypersurface with constant mean curvature, the ambient manifold $ M $ has constant sectional curvature and its the shape operator is $ A $ such that ${\nabla ^*}{T^{{A^2}}} = 0$ and $ A^2 $ is bounded, Assume $ x_0 $ be a fixed point and define $r(x): = dist({x_0},x)$, then
\[0 \ge \frac{{{{\left( {{\Delta _{{A^2}}}r} \right)}^2}}}{{(n - 1)\delta _n^2}} + {\partial _r}.{\Delta _{{A^2}}}r + Ric({\partial _r},{A^2}{\partial _r}) - \left\langle {{\nabla _{{\partial _r}}}div({A^2}),{\partial _r}} \right\rangle  - 2\sum\nolimits_i {\left\langle {{T^{{\nabla _{{\partial _r}}}A}}\left( {{e_i},{\partial _r}} \right),A{e_i}} \right\rangle .} \]
\end{theorem}
\begin{proof} A simple computation shows that when the ambient manifold has costant sectional curvature, then
\[{\partial _r}.{\partial _r}.Trace({A^2}) = Hess\left( {Trace({A^2})} \right)\left( {{\partial _r},{\partial _r}} \right) = 2\left\langle {{\nabla _{{\partial _r}}}\left( {div\left( {{A^2}} \right) - AdivA} \right),{\partial _r}} \right\rangle. \]
By assumption, we know the mean curvature of the hypersurface is constant,\linebreak so $\nabla Trace(A) = divA = 0$, Also note $Trace\left( {{{\left( {{\nabla _{{\partial _r}}}A} \right)}^2}} \right) \ge 0$. So the result follows by Proposition \ref{application2}.
\end{proof}
In the following Proposition we compare the extended Ricci tensor with the Ricci tensor of minimal hypersurfaces in euclidian spaces.
\begin{prop} Let ${\Sigma ^n} \subset {M^{n + 1}}(0)$ be a minimal Riemannian hypersurface then
\[\frac{1}{{\delta _n^2}}Ric\left( {X,{A^2}X} \right) - Ri{c_\Sigma }\left( {X,X} \right) \ge 0\]
\end{prop}
\begin{proof} Let $ B = \frac{1}{{\delta _n^2}}{A^2} $, by computation we have
\begin{eqnarray*}
   {Ric\left( {X,BX} \right) - Ri{c_\Sigma }\left( {X,X} \right)} &=& {   - \left\langle {X,{A^2}BX} \right\rangle  + \left\langle {X,{A^2}X} \right\rangle }   \\
&= & {  \left\langle {X,{A^2}\left( {1 - B} \right)X} \right\rangle  = \frac{1}{{\delta _n^2}}\left\langle {X,{A^2}\left( {\delta _n^2 - {A^2}} \right)X} \right\rangle  \ge 0}.
\end{eqnarray*}
\end{proof}


\begin{thebibliography}{99}
\bibitem{abresch1990complete}  U. Abresch,  D. Gromoll, On complete manifolds with nonnegative Ricci curvature,  Journal of the American
Mathematical Society, 3(2) (1990), 355-374.
\bibitem{alencar2015eigenvalue} H.  Alencar, GS. Neto, D.  Zhou,  Eigenvalue estimates for a class of elliptic differential operators on compact
manifolds, Bulletin of the Brazilian Mathematical Society, New Series, 46(3)(2015), 491-514.
\bibitem{besse2007einstein} A.   Besse,  Einstein manifolds. Springer Science,  Business Media; 2007.
\bibitem{brighton2013liouville} K. Brighton,   A Liouville-type theorem for smooth metric measure spaces,  Journal of Geometric Analysis,
23(2)(2013), 562-570.
\bibitem{cai1991ends} M.  Cai,  Ends of Riemannian manifolds with nonnegative Ricci curvature outside a compact set, Bulletin of
the American Mathematical Society, 24(2)(1991), 371-377.
\bibitem{calabi1958extension} E.  Calabi, et al, An extension of E. Hopfs maximum principle with an application to Riemannian geometry.
Duke Mathematical Journal, 25(1)(1958), 45-56.



\bibitem{cheeger2004degeneration} J. Cheeger,  Degeneration of Riemannian Metrics under Ricci Curvature Bounds, Lezioni Fermiane, [Fermi
Lectures], Scuola Normale Superiore, Pisa, 2001. MR2006642 (2004j: 53049),2004.
\bibitem{cheeger1971splitting} J.  Cheeger, D.  Gromoll, et al,  The splitting theorem for manifolds of nonnegative Ricci curvature. Journal
of Differential Geometry, 6(1)(1971), 119-128.
\bibitem{crandall1983viscosity} M.  Crandall,P. Lions,  Viscosity solutions of Hamilton-Jacobi equations,  Transactions of the American
mathematical society, 277(1)(1983),1-42.
\bibitem{daicomparison} X.  Dai, G. Wei,  Comparison geometry for Ricci curvature. preprint http://math ucsb edu/ dai/Ricci-book
pdf http://imrn oxfordjournals org/Downloaded from.
\bibitem{fang2009two} F.  Fang, X.  Li,  Z. Zhang,  Two generalizations of Cheeger-Gromoll splitting theorem via Bakry-\'{E}mery Ricci
curvature,  Annales de l’institut Fourier,  59( 2009),   563-573.

\bibitem{gilbarg2015elliptic} D. Gilbarg, N. Trudinger,  Elliptic partial differential equations of second order,  springer, 1983.
\bibitem{gomes2016eigenvalue}J.  Gomes, J.  Miranda,  Eigenvalue estimates for a class of elliptic differential operators in divergence form,
arXiv preprint arXiv:160700066. 2016.
\bibitem{hormander1994notions} L.  H\"{o}rmander,  Notions of convexity, volume 127 of Progress in Mathematics,  Birkh\"{a}user Boston Inc.,
Boston, MA; 1994.
\bibitem{hu2008bounds}  Z. Hu,  S.  Xu, Bounds on the fundamental groups with integral curvature bound, Geometriae Dedicata, 134 (2008), 1- 16.
\bibitem{jaramillo2015fundamental} M. Jaramillo,  Fundamental groups of spaces with Bakry–Emery Ricci tensor bounded below. The Journal
of Geometric Analysis. 25(3)(2015), 1828-1858.
\bibitem{kennard2014positive} L.   Kennard, W. Wylie,   Positive weighted sectional curvature,  arXiv preprint arXiv:14101558. 2014.
\bibitem{lions1983optimal} P. Lions,  Optimal control of diffusion processes and Hamilton–Jacobi–Bellman equations part 2: viscosity
solutions and uniqueness, Communications in partial differential equations, 8(11)(1983),1229-1276.
\bibitem{lott2003some} J.  Lott,  Some geometric properties of the Bakry-\'{E}mery-Ricci tensor, Commentarii Mathematici Helvetici,
78(4)(2003), 865-883.
\bibitem{petersen1998integral} P. Petersen, C. Sprouse, Integral curvature bounds, distance estimates, and
applications,  Jour. Differential Geom., 50(2)(1998), 269–298.
\bibitem{pigola2008vanishing} S.  Pigola,  M. Rigoli,  A. Setti,  Vanishing and finiteness results in geometric analysis: a generalization of the
Bochner technique,  vol. 266. Springer Science, Business Media; 2008.
\bibitem{sormani2001loops} C. Sormani,   On loops representing elements of the fundamental group of a complete manifold with nonnegative Ricci curvature,  Indiana University mathematics journal,  50(4)(2001), 1867-1883.
\bibitem{sormani1998nonnegative} C. Sormani,   Nonnegative Ricci curvature, small linear diameter growth and finite generation of fundamental
groups, arXiv preprint math/9809133, 1998.

\bibitem{wang2012eigenvalue} L.   Wang, Eigenvalue estimate for the weighted p-Laplacian, Annali di Matematica Pura ed Applicata,
191(3)(2012), 539-550.
\bibitem{wang2018gradient} L.   Wang,  Gradient estimates on the weighted p-Laplace heat equation, Journal of Differential Equations,
264(1)(2018), 506-524.
\bibitem{wang2016lower} Y. Wang, H.  Li,  Lower bound estimates for the first eigenvalue of the weighted p-Laplacian on smooth
metric measure spaces, Differential Geometry and its Applications,  45(2016), 23-42.

\bibitem{wei2007comparison} G.  Wei,  W. Wylie,  Comparison geometry for the Bakry-\'{E}mery Ricci tensor. Journal of Differential Geometry,
83(2)(2009), 337-405.
\bibitem{wu2016counting} J.  Wu,   Counting ends on complete smooth metric measure spaces, Proceedings of the American Mathematical Society, 144(5)(2016), 2231-2239.
\bibitem{wu2016comparison} J.  Wu,  Comparison Geometry for Integral Bakry-\'{E}mery Ricci Tensor Bounds, The Journal of Geometric Analysis,(2016), 1-40.
\bibitem{wu2018myers} J.  Wu,  Myers’ type theorem with the Bakry-\'{E}mery Ricci tensor, Annals of Global Analysis and Geometry, 54(4)(2018), 541-549.
\bibitem{wylie2017warped}  W.  Wylie,  A warped product version of the Cheeger-Gromoll splitting theorem, Transactions of the American
Mathematical Society, 369(9)(2017), 6661-6681.

\bibitem{wylie2015sectional} W.  Wylie,  Sectional curvature for Riemannian manifolds with density, Geometriae Dedicata,
178(1)(2015),151-169.
\bibitem{yau1976some} S. Yau,   Some function-theoretic properties of complete Riemannian manifold and their applications to
geometry, Indiana University Mathematics Journal,  25(7)(1976), 659-670.
\bibitem{zhu1997comparison} S. Zhu,  The comparison geometry of Ricci curvature,  Comparison geometry (Berkeley, CA, 1993–94),
30(1997), 221-262.













\end{thebibliography}
\end{document}